\documentclass[11pt,reqno]{amsart}

\usepackage{mathrsfs}
\usepackage{amsfonts}
\usepackage{amssymb}
\usepackage{amsthm}
\usepackage{amsmath}
\usepackage{tikz}
\usepackage{geometry}
\usepackage{tabu}
\usepackage{float}
\usepackage{color}
\usepackage{cite}

\allowdisplaybreaks

\newtheorem{theorem}{Theorem}[section]
\newtheorem{proposition}[theorem]{Proposition}
\newtheorem{lemma}[theorem]{Lemma}

\makeatletter \@addtoreset{equation}{section} \makeatother

\def\na{\nabla}
\def\p{\partial}

\def\tilde{\widetilde}

\newcommand{\beq}{\begin{equation}}
\newcommand{\eeq}{\end{equation}}

\newcommand{\Rmnum}[1]{\expandafter\@slowromancap\romannumeral #1@}

\begin{document}

\title[2D compressible MHD system]
{Global well-posedness for 2D non-resistive compressible MHD system in periodic domain}
\author[Jiahong Wu and Yi Zhu]{Jiahong Wu$^{1}$ and Yi Zhu$^{2}$}

\address{$^1$ Department of Mathematics, Oklahoma State University, Stillwater, OK 74078, United States}

\email{jiahong.wu@okstate.edu}

\address{$^2$ Department of Mathematics, East China University of Science and Technology, Shanghai 200237,  P.R. China}

\email{zhuyim@ecust.edu.cn}

\date{}

\begin{abstract}
This paper focuses on the 2D compressible magnetohydrodynamic (MHD) equations without
magnetic diffusion in a periodic domain. We present a systematic approach to establishing
the global existence of smooth solutions when the initial data is close to a background
magnetic field. In addition, stability and large-time decay rates are also obtained.
When there is no magnetic diffusion, the magnetic field and the density
are governed by forced transport equations and the problem considered here is difficult. This paper
implements several key observations and ideas to maximize the enhanced dissipation
due to hidden structures and interactions. In particular, the weak smoothing and stabilization generated
by the background magnetic field and the extra regularization in the divergence part of the velocity
field are fully exploited.
Compared with the previous works, this paper appears to be the first to investigate such system on bounded domains and the first to solve this problem by pure energy estimates, which help reduce the complexity in other approaches. In addition, this paper combines the well-posedness with the precise large-time behavior, a strategy that can be extended to higher dimensions.
\end{abstract}

\maketitle

\section{introduction}

\subsection{Background and Main Result}
This paper focuses on the two-dimensional (2D) viscous and non-resistive compressible
magnetohydrodynamic (MHD) system in a periodic domain $\mathbb{T}^2=[-\pi, \pi]^2$,
\begin{equation}\label{mhd0}
	\begin{cases}
		\tilde{\rho}_t +\nabla\cdot(\tilde{\rho}u)=0,  \quad\quad  (t, x) \in \mathbb{R}^+ \times \mathbb{T}^2,\\
		\tilde{\rho}u_t + \tilde{\rho}u\cdot\nabla u - \mu \Delta u-\lambda \nabla\nabla\cdot u + \nabla P(\tilde{\rho}) = B\cdot\nabla B - \frac{1}{2} \nabla |B|^2,\\
		B_t+u\cdot\nabla B +B\nabla\cdot u=B\cdot\nabla u,\\
		\nabla\cdot B=0,\\
		\tilde{\rho}(0,x)=\tilde{\rho}_0(x), \qquad u(0,x)=u_0(x), \qquad B(0,x)=B_0(x),
	\end{cases}
\end{equation}
where $\tilde{\rho} = \tilde{\rho} (t,x) \in \mathbb{R}^{+}$ represents the density of the fluid, $u = u(t,x)$ and $B =B(t,x)$ the velocity and the magnetic field, respectively. $P=P(\tilde{\rho}) \in \mathbb{R}^{+}$ is the scalar pressure which is a strictly increasing function of $\tilde{\rho}$. Here $\lambda$ and $\mu$ are viscous coefficients satisfying
$$ \mu >0, \qquad \lambda +\mu >0.$$

\vskip .1in
The goal here is to understand the well-posedness and stability problem on perturbations near
a background magnetic field. This study is partially motivated by a remarkable
stabilizing phenomenon observed in physical experiments performed on electrically
conducting fluids. These experiments have revealed that a background magnetic field can
actually stabilize MHD flows (see, e.g., \cite{Alex, Bur, Davi0, Davi1, Davi, Gall, Gall2}). We intend to fully understand the mechanism
and establish the observed stabilizing phenomenon as mathematically rigorous results on the
compressible MHD equations.

\vskip .1in
The MHD systems reflect the basic physics laws governing the motion of electrically conducting fluids such as plasmas, liquid metals, and electrolytes. The velocity field obeys the Navier-Stokes equations with
Lorentz forcing generated by the magnetic field while the magnetic field satisfies the Maxwell's equations of electromagnetism. The MHD equations have played key roles in the study of many phenomena in geophysics,
astrophysics, cosmology and engineering (see, e.g., \cite{Bis,Davi, Pri}).

\vskip .1in
Without loss of generality, we assume that $\mu = 1, \lambda= 0$ and $P'(1) = 1$. It is clear that a special solution of \eqref{mhd0} is given by the zero velocity field, the constant $1$ density and the background magnetic field $B_0=e_2$, where $e_2=(0,1)$.
The perturbation $(\rho, u, b)$  around this equilibrium with $\rho = \tilde{\rho}-1$ and $b= B - e_2$ then obeys
\begin{equation}\label{mhd1}
	\begin{cases}
		\rho_t +\nabla\cdot u = - \nabla\cdot(\rho u),  \qquad\qquad \qquad\qquad \qquad \qquad (t,x) \in \mathbb{R}^+ \times \mathbb{T}^2, \\
		u_t - \Delta u  - \left(\begin{array}{c}
			\nabla^{\bot} \cdot b \\
			0 \\
		\end{array}\right) +  \frac{1}{\rho  + 1}\nabla P = - u\cdot \nabla u - \frac{\rho}{\rho + 1}  \Delta u \\
		\qquad \qquad  +   \frac{1}{\rho + 1} b\cdot\nabla b- \frac{1}{2(\rho+1)}\nabla |b|^2 - \frac{\rho}{\rho + 1}\left(
		\begin{array}{c}
			\nabla^{\bot} \cdot b \\
			0 \\
		\end{array}
		\right), \\
		b_t - \nabla^{\bot}u_1=- u\cdot\nabla b + b\cdot\nabla u - b\nabla\cdot u,\\
		\nabla\cdot b=0,
	\end{cases}
\end{equation}
where $\nabla^{\bot}=(\partial_2,-\partial_1)$. In comparison with the original MHD system, (\ref{mhd1}) contains two extra terms
$
-\frac1{\rho+1} \left(\begin{array}{c}
	\nabla^{\bot} \cdot b \\
	0 \\
\end{array}\right)
$  and $-\nabla^{\bot}u_1$ on the left-hand sides of $u$ and $b$, respectively. The first term is further separated into the dominant linear term $ - \left(\begin{array}{c}
	\nabla^{\bot} \cdot b \\
	0 \\
\end{array}\right)$ and the small nonlinear term $- \frac{\rho}{\rho + 1}\left(
\begin{array}{c}
\nabla^{\bot} \cdot b \\
0 \\	\end{array}
\right)$ on the right of the velocity equation. These two terms, resulting from the expansion near the background magnetic field, help enhance the dissipation and the stabilization in the system (\ref{mhd1}).

\vskip .1in
Due to the lack of diffusion or damping mechanism in the equation of the magnetic field and density,
the well-posedness and stability problem considered here appears to be extremely challenging. Standard direct approaches
would not solve this problem. This paper presents several key observations and new ideas to
exploit the enhanced dissipation due to the background magnetic field and the special
coupling structures within the MHD system (\ref{mhd1}). Our main result can be stated as follows.
We use the notation $A \lesssim B$, which means $A \leq CB$ for a universal constant $C>0$.

\begin{theorem}\label{thm1}
	Let $s\ge 10$ be an integer. Assume that the initial data $(\rho_0, u_0, b_0) \in H^s(\mathbb{T}^2)$ satisfies
	$$	\na \cdot b_0 =0, \qquad \int_{\mathbb T^2} b_0(x)\,dx =0.
	$$
	Then there exists a small constant $\epsilon>0$ such that, if
	\begin{equation*}
		\|\rho_0\|_{H^s} + \|u_0\|_{H^s} + \|b_0\|_{H^s} \leq \epsilon,
	\end{equation*}
	system \eqref{mhd1} admits a unique global classical solution $(\rho, u, b)$ in $H^s(\mathbb T^2)$. In addition, the solution $(\rho, u, b)$ admits the following upper bounds,
	for any $t>0$,
	\begin{align*}
		&\sup_{0\leq \tau \leq t} w_0(\tau) \big( \| u\|_{H^s}^2 +\| \rho\|_{H^{s}}^2 +\|b\|_{H^{s}}^2 \big) \lesssim \;\epsilon^2, \\
		&
		\|\rho(t)\|_{H^{s-1}}^2	+\|u(t)\|_{H^{s-1}}^2 + \|b(t)\|_{H^{s-1}}^2 + \int_0^t \| \nabla u\|_{H^{s-1}}^2 \; d\tau \lesssim \;\epsilon^2,\\
		& \sup_{0\leq \tau \leq t} w_k(\tau) \big( \| \partial_2 u\|_{H^{s-k}}^2 +\| \partial_2 \rho\|_{H^{s-k}}^2 +\|\partial_2 b\|_{H^{s-k}}^2 \big) \\
		&  \qquad\qquad \qquad + \int_0^t w_k (\tau) \|\nabla \partial_2 u\|_{H^{s-k}}^2 \; d\tau \lesssim \;\epsilon^2  , \qquad  k=1,2,3,\\
	&\int_0^t w_k(\tau) \| \partial_2^2  \rho\|_{H^{s-1-k}}^2 \; d\tau \lesssim \;\epsilon, \quad
	\int_0^t w_k(\tau) \| \partial_2^2 \nabla^{\bot}\cdot b\|_{H^{s-2-k}}^2 \; d\tau\lesssim \;\epsilon^2,\,\, k=1, 2, 3,\\
	&\sup_{0\leq \tau \leq t} w_k(\tau)\|\partial_1 u_1\|_{H^{s-k}}^2  + \int_0^t w_k(\tau)
	\|\nabla \partial_1 u_1\|_{H^{s-k}}^2  d\tau \lesssim \;\epsilon^2, \,\,\, k=1, 2, 3,
\end{align*}
where the time weight $w_k$ is given by
$$
	w_k(t)=(1+t)^{k-\sigma},       \qquad  k=-1,0,1,2,3,
$$
with $0< \sigma <\frac{1}{2}$ being an arbitrarily fixed constant.
\end{theorem}

As aforementioned, the background magnetic field and the hidden interactive structures of (\ref{mhd1})  generate
enhanced dissipation. The upper bounds stated in Theorem \ref{thm1} reflect the regularity and
decay properties due to the enhanced dissipation. We will give a more detailed account on the
mechanisms when we constructed the energy functional below.

\vskip .1in
We briefly describe related results. Since the initial study of Alfven \cite{Alf},
the well-posedness and stability problem on the MHD equations near
a background magnetic field has gained renewed interests and there has been substantial
developments recently. Most of the current results are for the incompressible MHD equations.
\cite{LXZ} and \cite{in7} initiated the stability study on the 2D and 3D incompressible viscous and non-resistive MHD systems and have since inspired many important further investigations (see, e.g.,
\cite{in1,DZ, in8, PZZ,in9, RXZ,Wu3}). Significant stability results have
also been obtained for the incompressible ideal MHD or fully dissipative MHD equations
\cite{in2,in3,in5,in11}. Small data
global well-posedness, stability and large-time behavior problems on the incompressible MHD equations
with various partial dissipation have also attracted considerable interest and a rich array of results have been developed (see, e.g., \cite{BLWu, CW1, CW2,FengFWu, JiWu,  in6, LinWu, WZ2, Wu4,Zhou2}).

\vskip .1in
There are relatively fewer results for the compressible MHD equations. Important well-posedness results on the compressible MHD equations with both viscosity and resistivity have been established
\cite{com1, com2, com3, com4}. When there is no magnetic diffusivity, the well-posedness problem
(even for small initial data) as well as the stability problem near a background magnetic field
becomes extremely difficult. Hu and Lin \cite{hu} were able to establish the well-posedness for
the 2D compressible MHD equations with a special class of data near a background magnetic field.
Two interesting results on the 1D compressible non-resistive  MHD equations are also available \cite{jiangzhang, liyang}. Compressible and incompressible non-resistive
MHD equations in a bounded domain are studied in \cite{tanwang}.
The work of Wu and Wu \cite{WuWu} presented a systematic approach to the stability problem on the 2D
compressible equations without magnetic diffusion in the whole space $\mathbb R^2$ case.
\cite{WuWu} exploited the extra stabilizing effects by converting the governing system into a system of wave equations and employed extensive Fourier analysis. It appears difficult to extend
the approach of \cite{WuWu} to periodic domains.
In addition, the method presented in this paper is different and appears to be much simpler.
Our strategy of combining the well-posedness problem with the precise large-time behavior of the solution appears to be necessary and efficient in solving this type of well-posedness problems. It is very hopeful that the approach of this paper can be extended to the three-dimensional case, namely  $\mathbb{R}^3$ or $\mathbb{T}^3$ case. There are some differences between 2D case and 3D case. For 2D case considered in this paper, when applying $\nabla$ on equations, there will appear at least one good part in nonlinear terms. For example, $\partial_1 u \cdot \nabla b = {\it \partial_1 u_1} \partial_1 b + \partial_1 u_2 {\it \partial_2 b}$ and ${\it\partial_2 u} \cdot \nabla b$ (coming from $\nabla u \cdot \nabla b$) always contain a strong dissipative part. However, this will not  hold for 3D case.

\vskip .1in
\subsection{Enhanced Dissipation and Construction of Energy Functional} We explain
the key ideas and observations in the proof of Theorem \ref{thm1}. The proof is not direct due to
the lack of diffusion or damping in the equations of $\rho$ and $b$. Enhanced dissipation needs to
be exploited in order to offset the potential growth of $\rho$ and $b$. More precisely, we implement
the following key ideas and observations.

\begin{itemize}
\item The first is that the perturbation of magnetic field near the equilibrium $e_2$ generates
extra dissipation in the $x_2$ direction. This reflects the smoothing and stabilizing phenomenon
observed in the aforementioned experiments. Mathematically the linearized system of (\ref{mhd1}) can be
converted into a system of wave equations, which exhibits dissipative and dispersive
regularizations. In fact, the linearization of (\ref{mhd1}) is given by
$$
\begin{cases}
	\p_t \rho = -\na\cdot u, \\
	\p_t u =\Delta u -\na \rho + \left(\begin{array}{c}
		\nabla^{\bot} \cdot b \\
		0 \\
	\end{array}\right), \\
	\p_t b = \na^\perp u_1,
\end{cases}
$$
where $\na \rho$ linearizes $\na P(\rho +1)$,
$$
\na P(\rho +1) = P'(\rho +1) \na \rho \approx \na \rho \quad\mbox{for $\rho\approx 0$ and $P'(1)=1$. }
$$
Differentiating in $t$ and making substitutions, we find
\begin{align*}
	&\p_{tt} u_1 -\Delta \p_t u_1-\Delta u_1 = \p_1\na\cdot u, \\
	&\p_{tt} u_2 -\Delta \p_t u_2 = \p_2\na\cdot u, \\
	&\p_{tt} \rho -\Delta \p_t \rho -\Delta \rho = - \p_1 \na^\perp\cdot b, \\
	&\p_{tt} b -\Delta \p_t b -\Delta b = - \p_1 \na^\perp\rho.
\end{align*}
Furthermore, the equations of $\rho$ and $b$ can be converted into the fourth-order wave equations
\begin{align*}
	& (\p_{tt} -\Delta \p_t -\Delta)^2 \rho  =\p_1^2 \Delta \rho,\\
	& (\p_{tt} -\Delta \p_t -\Delta)^2 b  = \p_1^2 \Delta b.
\end{align*}
The combined operator $\Delta^2- \p_1^2 \Delta = \p_2^2 \Delta$ reveals the regularization in the $x_2$-direction.

\item A simple combination of the linearized equations of $\rho$ and $u_2$,
\beq\label{ww}
\begin{cases}
	\p_t \rho = -\p_2 u_2 - \p_1 u_1, \\
	\p_t u_2 = - \p_2 \rho + \Delta u_2
\end{cases}
\eeq
exhibits the wave structure, which yields the vertical dissipation for $\rho$.

\item Even though the fluid of system (\ref{mhd1}) is compressible,  the divergence of the velocity field behaves better than $\nabla \times u$. As we shall see in later parts,  $\partial_1u_1+\partial_2u_2$ enjoys rapid decay properties.

\item The evolution of the combined quantity $\Omega$ defined by
 \beq\label{OO}
 \Omega \triangleq \nabla^{\bot} \cdot b - \partial_1 P - \frac{1}{2} \partial_1 |b|^2 + b \cdot \nabla b_1,
 \eeq
when coupled with the equation of $u_1$, generates enhanced dissipation. This extra regularization allows us to establish uniform bounds for the time integrals of $\Omega$ and $\p_1 u_1$. The introduction of $\Omega$ helps us control
the velocity nonlinear terms.
\end{itemize}

\vskip .1in
These ideas and observations is incorporated in the construction of the energy functional.
It consists of several layers of time-weighted energy functions with some associated with the linearized system
and some designed for the nonlinear terms.

~\\~
{\bf {Part 1: Energy functions associated with the linearized system}}
~\\~

This part defines several energy functionals to capture the dissipation in the $x_2$ direction.
These functionals are time-weighted to reflect the decay properties of various norms of
the solutions to (\ref{mhd1}). For an arbitrarily fixed constant $0< \sigma <\frac{1}{2}$ we define
the time weight $w_k(t)$ by
$$ w_k(t)=(1+t)^{k-\sigma},  \qquad k=-1,0,1,2,3.
$$
We first define the basic energy $\mathcal{E}_0 $ as
\begin{equation}\label{E0}
\begin{split}
  \mathcal{E}_0(t) \triangleq &   \sup_{0\leq \tau \leq t} w_0(\tau) \big( \| u\|_{H^s}^2 +\| \rho\|_{H^{s}}^2 +\|b\|_{H^{s}}^2 \big) \\
  &+ \int_0^t w_{-1} (\tau)( \| u\|_{H^s}^2 +\| \rho\|_{H^{s}}^2 +\|b\|_{H^{s}}^2 \big) \; d\tau + \int_0^t w_0 (\tau)\|\nabla u\|^2_{H^s} \; d\tau .
\end{split}
\end{equation}
The dissipative energy piece involves $(\p_2 u, \p_2 \rho, \p_2 b)$,
\begin{align*}
  \mathcal{E}_k(t) \triangleq &   \sup_{0\leq \tau \leq t} w_k(\tau) \big( \| \partial_2 u\|_{H^{s-k}}^2 +\| \partial_2 \rho\|_{H^{s-k}}^2 +\|\partial_2 b\|_{H^{s-k}}^2 \big) \\
 &  + \int_0^t w_k (\tau) \|\nabla \partial_2 u\|_{H^{s-k}}^2 \; d\tau  , \qquad  k=1,2,3.
\end{align*}
It's natural to include the corresponding time integral pieces of $\p_2^2\rho$ and $\p_2^2 b$,
\begin{align*}
	&\mathcal{P}_k(t) \triangleq  \int_0^t w_k(\tau) \| \partial_2^2  \rho\|_{H^{s-1-k}}^2 \; d\tau , \qquad k=1, 2, 3,  \\
		&\mathcal{B}_k(t)\triangleq \int_0^t w_k(\tau) \| \partial_2^2 \nabla^{\bot}\cdot b\|_{H^{s-2-k}}^2 \; d\tau , \qquad k=1,2, 3.
\end{align*}

\vskip .1in
It is clear that $\mathcal{E}_2(t)$ can be interpolated via $\mathcal{E}_1(t) $
and $\mathcal{E}_3(t)$. We will not include $\mathcal{E}_2(t)$ in our final energy functional.
It is defined and later used to simplify our presentation. Similarly  $\mathcal{P}_2(t) $, $\mathcal{B}_2(t) $ as well as $\mathcal{F}_2(t) $, $\mathcal{A}_2(t)$ defined below are intermediate energy functionals.

~\\~
{\bf {Part 2: A combined quantity $\Omega$ and energy functions for the nonlinear terms}}
~\\~

Due to the aforementioned wave structures, we expect $u_1$ to behave well in the $x_1$ direction.
We define a combined quantity $\Omega$ to include all terms $\rho$ and $b$ in the equation
of $u_1$, namely
$$
\Omega \triangleq \nabla^{\bot} \cdot b - \partial_1 P - \frac{1}{2} \partial_1 |b|^2 + b \cdot \nabla b_1.
$$
Then the equation of $u_1$ can be written as
\begin{equation} \label{equ1}
  \begin{split}
    \partial_t u_1 + u \cdot \nabla u_1 - \Delta u_1 - \Omega = -\frac{\rho}{\rho+1}\big(\Delta u_1 + \Omega \big).
  \end{split}
\end{equation}
The evolution of $\Omega$ can be derived by considering the evolution of each term in $\Omega$. According to \eqref{mhd1}, we find
\begin{equation}\nonumber
\begin{split}
(\nabla^{\bot} \cdot b)_t + u \cdot \nabla (\nabla^{\bot} \cdot b )- \Delta u_1 =& - \partial_2 u\cdot \nabla b_1 + \partial_1 u\cdot \nabla b_2 +  \nabla^{\bot}\cdot \big(  b \cdot \nabla u - b \nabla \cdot u \big),\\
\big(\frac{1}{2}\partial_1 |b|^2 \big)_t + \frac{1}{2} u \cdot \nabla (\partial_1 |b|^2 )=& \; \frac{1}{2} \partial_1 u \cdot \nabla (|b|^2 )+ \partial_1 \big( b \cdot \nabla^{\bot} u_1 + b \cdot \big( b \cdot \nabla u - b \nabla \cdot u \big) \big), \\
(b\cdot \nabla b_1)_t + u \cdot \nabla (b \cdot \nabla b_1) =  &\; b \cdot \nabla (\partial_2 u_1 )+ \nabla^{\bot} u_1 \cdot \nabla b_1 - b \cdot \nabla (u \cdot \nabla b_1 )\\
& + b \cdot \nabla (b \cdot \nabla u_1 - b_1 \nabla \cdot u) + (b \cdot \nabla u - b \nabla \cdot u) \cdot \nabla b_1, \\
(\partial_1 P)_t + u\cdot \nabla (\partial_1 P) + \partial_1 \nabla \cdot u =& - \partial_1 u \cdot \nabla P - \partial_1 \big(P'(\tilde \rho) \tilde \rho - 1 \big) \nabla \cdot u.
\end{split}
\end{equation}
Without loss of generality, we have assumed that $P'(1)=1$. Combining these equations, we obtain
\begin{equation}\label{eqomega}
\begin{split}
&\; \Omega_t + u \cdot \nabla \Omega - 2 \partial_1^2 u_1 =  \;\partial_2^2 u_1 + \partial_1\partial_2 u_2 \\
&\qquad + \partial_1 u\cdot \nabla b_2 - \frac{1}{2} \partial_1 u \cdot \nabla (|b|^2 )+ \partial_1 u \cdot \nabla P + \nabla^{\bot} \cdot (b \cdot \nabla u)\\
&\qquad - \partial_2 u \cdot \nabla b_1 - \partial_1 (b \cdot \nabla^{\bot} u_1)   + b \cdot \nabla (\partial_2 u_1) + \nabla^{\bot} u_1 \cdot \nabla b_1\\
&\qquad  + b \cdot \nabla(b \cdot \nabla u_1)- \partial_1 \big( b \cdot (b \cdot \nabla u) \big)- \nabla^{\bot} \cdot (b \nabla \cdot u) + \partial_1 (|b|^2 \nabla \cdot u)  \\
&\qquad  + \partial_1  \Big\{  \big(P'(\tilde \rho) \tilde \rho - 1 \big) \nabla \cdot u \Big\}  - \nabla \cdot u (b \cdot \nabla b_1) - b \cdot \nabla (b_1 \nabla \cdot u).
\end{split}
\end{equation}
The equation of $\Omega$ is quite lengthy and we still adopt the strategy of focusing on the linearized system,
$$
\begin{cases}
	\partial_t u_1  = \Delta u_1 +  \Omega,\\
\partial_t \Omega = 2 \partial_1^2 u_1 + \partial_2^2 u_1 + \partial_1\partial_2 u_2.
\end{cases}
$$
This system provides extra regularization for $\p_1 u_1$ and $\Omega$ via its wave structure. The
following functionals are designed to incorporate these regularity and decay properties.
\begin{align*}
& \tilde{\mathcal{A}} (t) \triangleq \int_0^t w_0(\tau)  \|\Omega\|_{H^{s-1}}^2 \; d\tau,\\
&\mathcal{F}_k(t) \triangleq   \sup_{0\leq \tau \leq t} w_k(\tau) \big(\|\partial_1 u_1\|_{H^{s-k}}^2 + \|\Omega\|_{H^{s-k}}^2 \big)\\
 & \qquad\qquad + \int_0^t w_k(\tau)
\|\nabla \partial_1 u_1\|_{H^{s-k}}^2 \; d\tau  , \qquad k=1, 2, 3,\\
&\mathcal{A}_k(t)  \triangleq \int_0^t w_k(\tau) \| \Omega \|_{H^{s-k}}^2 \; d\tau, \qquad k=1,2, 3.
\end{align*}
$\tilde{\mathcal{A}} (t)$ and $\mathcal{A}_k(t)$ are defined to assist in exploiting the dissipation of $\partial_1 u_1$. It is then clear that ${\mathcal E}_k(t)$ and ${\mathcal F}_k(t)$ covers the regularization of $\nabla \cdot u$.

\vskip .1in
Even though the highest-order energy \eqref{E0} may grow in time, we expect uniform boundedness for some lower-order energy. Therefore, we define
\begin{equation*}
\mathfrak{E}(t) \triangleq \sup_{0\leq \tau \leq t} \big(\|u\|_{H^{s-1}}^2 + \|b\|_{H^{s-1}}^2 + \|\rho\|_{H^{s-1}}^2 \big) + \int_0^t \|\nabla u\|_{H^{s-1}}^2 \; d\tau.
\end{equation*}
Finally, to reflect the fact that any lower-order Sobolev's norm of $u_2$ decays in time, we set
\begin{equation*}
\mathfrak{U}(t) \triangleq  \sup_{0\leq \tau \leq t} w_2 (\tau) \|u_2\|_{\dot H^{s-2}}^2 + \int_0^t w_2 (\tau) \|\nabla u_2\|_{\dot H^{s-2}}^2 \; d\tau.
\end{equation*}
The energy functionals $\mathcal{F}_k(t), \mathcal{A}_k(t), \mathfrak{E}(t)$ and $\mathfrak{U}(t)$ are sufficient in handling the nonlinear terms in system \eqref{mhd1}.

~\\~
{\bf {Part 3: Summary of energy functionals and the total energy}}
~\\~

We summarize the energy functionals defined in the previous two parts in two tables.
The first table classifies these functionals according to their purposes: some of them for
the linearized system and some for the nonlinear terms.

\begin{table}[H]

\begin{center}

\begin{tabu} to 0.8  \textwidth{X[1,c]| X[1,c]}

\hline

For Linear Terms & For Nonlinear Terms \\

\hline

$\mathcal{E}_0 $ &    $\mathcal{F}_1$, $\mathcal{F}_3$       \\

$\mathcal{E}_1$, $\mathcal{E}_3$  &  $\mathcal{A}_1$, $\mathcal{A}_3$  \\

$\mathcal{P}_1$, $\mathcal{P}_3$  &   $\tilde{\mathcal{A}}$\\

$\mathcal{B}_1$, $\mathcal{B}_3$  & $\mathfrak{E},\; \mathfrak{U} $    \\

\hline
\end{tabu}

\end{center}

\end{table}

The intermediate energy functionals $\mathcal{E}_2(t) $, $\mathcal{P}_2(t) $, $\mathcal{B}_2(t) $, $\mathcal{F}_2(t) $ and $\mathcal{A}_2(t) $ are omitted from this table.

\vskip .1in
The second table categorizes the energy functionals according to their properties

\begin{table}[H]

\begin{center}

\begin{tabu} to 0.9 \textwidth{X[1,c]| X[3,c]|X[3,c]| X[2,c]}
	
\hline

Basic Energy & Strong Dissipation on Vertical Derivative & Dissipative Structure on Divergence Part & Decay estimate of $u_2$  \\

\hline

$\mathcal{E}_0 $ &    $\mathcal{E}_1$, $\mathcal{E}_3$    &  $\mathcal{F}_1$, $\mathcal{F}_3$  &  $ \mathfrak{U} $ \\

$\mathfrak{E}$ & $\mathcal{P}_1$, $\mathcal{P}_3$ & $\mathcal{A}_1$, $\mathcal{A}_3$  & \; \\

\; & $\mathcal{B}_1$, $\mathcal{B}_3$  &     $\tilde{\mathcal{A}}$         & \; \\

\hline

\end{tabu}

\end{center}

\end{table}

The two tables reveal the intrinsic connection of the energy functionals. The total energy is then
defined to be
\begin{align*}
	  \mathcal{E}_{total}(t)\triangleq & \sum_{i=0,1,3} \mathcal{E}_i(t)+\sum_{i=1,3} \mathcal{P}_i(t)+\sum_{i=1,3} \mathcal{B}_i(t)+\sum_{i=1,3} \mathcal{F}_i(t)+\sum_{i=1,3} \mathcal{A}_i(t) \\
	& \qquad  +  \tilde{\mathcal{A}}(t)  +  \mathfrak{U}(t)+\mathfrak{E}(t).
\end{align*}
Our main efforts are devoted to proving the estimate
\beq\label{bb}
\mathcal{E}_{total}(t) \le C^* \,\mathcal{E}_{total}(0) + C^*\, \mathcal{E}^{\frac32}_{total}(t),
\eeq
where $C^*$ is an absolute constant. An application of the bootstrapping argument
to (\ref{bb}) then yields the desired global \textit{a priori} upper bounds of Theorem \ref{thm1}.
The proof of (\ref{bb}) is extremely lengthy and is accomplished in four sections.

\vskip .1in
The rest of this paper is divided into six sections. Section \ref{pre} presents two tool lemmas
to be used in the proof of (\ref{bb}). Section \ref{basic} establishes the upper bounds for
the energy functionals $\mathcal{E}_0(t)$ and $\mathfrak{E}(t)$ while Section \ref{Strong}
estimates  $\mathcal{E}_k(t)$,  $\mathcal{P}_k(t)$ and  $\mathcal{B}_k(t)$. Section \ref{divp}
focuses on the bounds for  $\tilde{\mathcal{A}}$, $\mathcal{A}_k$ and $\mathcal{F}_k$, and Section
\ref{u2} shows the decay estimates of $u_2$. Section \ref{pp} completes the proof of Theorem \ref{thm1}.

\vskip .3in
\section{Preliminaries}
\label{pre}

This section presents two lemmas to be used in the proof of the propositions in the subsequent
sections. The first lemma provides an upper bound on a commutator associated with
transport equations while the second lemma bounds a triple product arising naturally
in the control of nonlinear terms in transport equations. Throughout the rest of this paper, $\p^s$ with a positive integer $s$ denotes the partial derivative $\partial_1^{\alpha_1}\partial_2^{\alpha_2}$ with $\alpha_1 + \alpha_2 =s$.

\begin{lemma}\label{prop0}
Let $s_0\ge 1$ be an integer. Assume $f\in H^{s_0}(\mathbb{T}^2)$ and $g\in H^{s_0-1}(\mathbb{T}^2)\cap H^{[\frac{s_0}{2}] + 2}$. Write
$$
[\partial^{s_0}, f] g \triangleq \partial^{s_0}(fg) - f \partial^{s_0} g.
$$
Then,
\begin{equation*}
\big\|[\partial^{s_0}, f] g \big\|_{L^2} \lesssim  \big\|\nabla  f\big\|_{H^{[\frac{s_0}{2}] + 1}} \big\|g \big\|_{H^{s_0 - 1}} + \big\| \nabla f\big\|_{H^{s_0 - 1}} \big\|g\big\|_{H^{[\frac{s_0}{2}] + 2}}.
\end{equation*}
\end{lemma}

\begin{proof}
By Leibniz formula,
\begin{equation*}
\begin{split}
[\partial^{s_0}, f] g =  \partial^{s_0}(fg) - f \partial^{s_0} g =
\sum_{k = 1}^{s_0} C_{s_0}^k \partial^k f \partial^{s_0 - k}g.
\end{split}
\end{equation*}
It then follows from H\"{o}lder's inequality that
\begin{equation*}
\begin{split}
\big\|[\partial^{s_0}, f] g\big\|_{L^2} \lesssim & \;
\sum_{k = 1}^{[s_0 / 2]}  \big\|\partial^k f \partial^{s-k} g\big\|_{L^2} +
\sum_{k = [\frac{s_0}{2}] + 1}^{s_0} \big\| \partial^k f \partial^{s-k}g \big\|_{L^2} \\
\lesssim &  \; \sum_{k = 1}^{[s_0 / 2]}  \big\|\partial^k f \big\|_{L^\infty} \big\|\partial^{s-k} g \big\|_{L^2} +
\sum_{k = [\frac{s_0}{2}] + 1}^{s_0} \big\| \partial^k f \big\|_{L^2} \big\|\partial^{s-k}g \big\|_{L^\infty} \\
\lesssim &  \; \|\nabla  f \|_{H^{[\frac{s_0}{2}] + 1}} \|g\|_{H^{s_0 - 1}} + \| \nabla f\|_{H^{s_0 - 1}} \|g\|_{H^{[\frac{s_0}{2}] + 2}}.
\end{split}
\end{equation*}
This completes the proof of the lemma.
\end{proof}

A special consequence of Lemma \ref{prop0} is the following anisotropic upper bound for the integral of a triple product.

\begin{lemma}\label{prop1}
Let $s_0 \ge 1$ be an integer. Let $f=(f_1, f_2)$ and $g=(g_1, g_2)$ be two smooth vector fields on $\mathbb{T}^2$. Then the following anisotropic inequality holds,
\begin{equation}\nonumber
\begin{split}
&\Big|\int_{\mathbb{T}^2} \partial^{s_0} (f \cdot \nabla g ) \partial^{s_0} g \; dx \Big| \\
\lesssim & \;
\Big(\|\nabla f_1\|_{H^{[\frac{s_0}{2}] + 1}} \|\partial_1 g\|_{H^{s_0-1}} +\|\nabla f_2\|_{H^{[\frac{s_0}{2}] + 1}} \|\partial_2 g\|_{H^{s_0-1}} \Big)\|g\|_{\dot H^{s_0}} \\
&+\Big(\|\nabla f_1\|_{H^{s_0-1}} \| \partial_1 g\|_{H^{[\frac{s_0}{2}] + 2}}+ \|\nabla f_2\|_{H^{s_0-1}} \| \partial_2 g\|_{H^{[\frac{s_0}{2}] + 2}}\Big) \|g\|_{\dot H^{s_0}} \\
& + \|\nabla \cdot f\|_{H^2}\|g\|_{\dot H^{s_0}}^2.
\end{split}
\end{equation}
\end{lemma}

\vskip .3in
\section{Upper bounds for the basic energy functionals}
\label{basic}

This section establishes the desired \textit{a priori} upper bounds for the
basic energy $\mathcal{E}_0(t)$ and $\mathfrak{E}(t)$. Since these two energy functionals share
similar structures, we put their estimates together. The upper bounds are stated in the following
proposition. Without loss of generality, we assume that
\begin{equation}\label{assum}
	\mathfrak{E}(t) \leq 1,  \quad  \| \rho \|_{L^{\infty}} \leq \frac{1}{2}\quad \text{for} \quad  t \in [0,T].
\end{equation}
The proof of Theorem \ref{thm1} is based on the bootstrapping argument, so (\ref{assum}) can be regarded as part of the ansatz. In addition, since the initial data is small and the solution is also small
in $H^{s-1}$, (\ref{assum}) does not impose any extra constraint on the solution we are seeking.

\begin{proposition}\label{lemmaBasic}
Let $\mathcal{E}_{total}(t)$ be the total energy functional defined in the introduction, namely
\begin{equation}\label{total}
\begin{split}
  \mathcal{E}_{total}(t)\triangleq & \sum_{i=0,1,3} \mathcal{E}_i(t)+\sum_{i=1,3} \mathcal{P}_i(t)+\sum_{i=1,3} \mathcal{B}_i(t)+\sum_{i=1,3} \mathcal{F}_i(t)+\sum_{i=1,3} \mathcal{A}_i(t) \\
 & \qquad  +  \tilde{\mathcal{A}}(t)  +  \mathfrak{U}(t)+\mathfrak{E}(t).
\end{split}
\end{equation}
Then the following inequalities hold for all positive time $t$,
\begin{equation}\nonumber
   \mathcal{E}_0(t)  \lesssim  \; \mathcal{E}_0(0) +  \mathcal{E}_{total}^{\frac{3}{2}}(t), \qquad \mathfrak{E}(t)  \lesssim \; \mathcal{E}_0(0) +  \mathcal{E}_{total}^{\frac{3}{2}}(t).
\end{equation}
\end{proposition}

\begin{proof}
We start with the $L^2$ estimate. Taking the $L^2$-inner product (on $\mathbb{T}^2$) of $(\rho, \tilde{\rho}u, b)$ with the corresponding equations in \eqref{mhd1}, we obtain
\begin{equation}\label{I123}
\begin{split}
  \frac{1}{2} \frac{d}{dt} \Big( \|\sqrt{\tilde \rho} u\|_{L^2}^2 + \|\rho\|_{L^2}^2 + \|b\|_{L^2}\Big) + \|\nabla u\|_{L^2}^2  = \sum_{i = 1}^3 I_i,
\end{split}
\end{equation}
where
\begin{equation}\nonumber
  \begin{split}
I_1 = & -\int_{\mathbb{T}^2} (\nabla \cdot u ) \rho\; dx - \int_{\mathbb{T}^2}  \nabla \rho\cdot
 u\; dx  +  \int_{\mathbb{T}^2}  \nabla^{\bot} \cdot b \; u_1  \\
& + \int_{\mathbb{T}^2}  \nabla^{\bot}  u_1  \cdot b \; dx  + \int_{\mathbb{T}^2} b \cdot \nabla b  \cdot u \; dx + \int_{\mathbb{T}^2} b \cdot \nabla u  \cdot b \; dx \\
& - \int_{\mathbb{T}^2} b \; \nabla \cdot u  \cdot b \; dx- \int_{\mathbb{T}^2} \frac{1}{2}(\nabla |b|^2)\cdot  u \; dx- \int_{\mathbb{T}^2}  u \cdot \nabla b \cdot b\; dx, \\
I_2 = & - \int_{\mathbb{T}^2}  u \cdot \nabla \rho  \; \rho \; dx
- \int_{\mathbb{T}^2} \rho \; \nabla \cdot u \; \rho \; dx, \qquad I_3 =  \int_{\mathbb{T}^2}  (\nabla \rho -  \nabla P) \cdot u \; dx.
\end{split}
\end{equation}
By integration by parts and $\na\cdot b=0$,
$$
I_1 =0.
$$
Integration by parts and H\"{o}lder's inequality yield
$$
 |I_2| \lesssim  \; \|\nabla \cdot u\|_{L^\infty} \|\rho\|_{L^2}^2.
$$
To bound $I_3$, we integrate by parts and set
\begin{equation*}
	q(\rho)\triangleq (P(\rho+1) -P(1))-\rho = \int_0^\rho \big(P'(r + 1) - 1 \big) \; dr.
\end{equation*}
to rewrite $I_3$ as
$$
I_3 = -\int_{\mathbb T^2} q(\rho) \na\cdot u\,dx.
$$
$q(\rho)$ is a smooth function of $\rho$ with  $q(0) = q'(0) = 0$ (due to $P'(1)=1$),
thus $q(\rho) \sim \rho^2$ for small $\rho$, say $\rho<1$. Therefore,
$$
    |I_3| \lesssim  \; \|q(\rho)\|_{L^1} \|\nabla \cdot u\|_{L^\infty}  \lesssim \; \|\nabla \cdot u\|_{L^\infty} \|\rho\|_{L^2}^2.
$$
Since $\tilde \rho$ perturbs around a non-trivial equilibrium, i.e., $\tilde \rho \sim 1$, we have $\|\sqrt{\tilde \rho} u\|_{L^2}^2 \sim \|u\|_{L^2}^2$.
Integrating \eqref{I123} in time on $[0, t]$ yields
\begin{equation}\label{i123}
\begin{split}
  \sup_{0\leq \tau \leq t} \big( \|u\|_{L^2}^2 + \|\rho\|_{L^2}^2 + \|b\|_{L^2}^2 \big) + &\int_0^t \|\nabla u\|_{L^2}^2 \; d\tau \lesssim  \;\sup_{0\leq \tau \leq t} w_0 \|\rho\|_{L^2}^2 \int_0^t  w_0^{-1} w_3^{-\frac{1}{2}} w_3^{\frac{1}{2}}   \|\nabla \cdot u\|_{L^\infty} \; d\tau \\
 &  \lesssim  \; \mathcal{E}_0 (0)  + \mathcal{E}_0 (t) \big(\mathcal{E}_3^\frac{1}{2} (t) + \mathcal{F}_3^\frac{1}{2} (t)\big).
\end{split}
\end{equation}
Here we have used the integrability of $w_0^{-2}(\tau) w_3^{-1}(\tau) = (1+\tau)^{3\sigma-3} \leq  (1+\tau)^{-\frac{3}{2}} $ for $\sigma < \frac{1}{2}$. This finishes the $L^2$ estimate for $\mathfrak{E}(t)$.

The $L^2$ estimate of $\mathcal{E}_0(t)$ follows from a similar process. Multiplying \eqref{I123}
by the time weight $w_0$ leads to
\begin{equation}\label{I1234}
\begin{split}
  \frac{1}{2} \frac{d}{dt} w_0 \Big( \|\sqrt{\tilde \rho} u\|_{L^2}^2 + \|\rho\|_{L^2}^2 &+ \|b\|_{L^2}^2 \Big) + \frac{\sigma}{2}w_{-1}\Big( \|\sqrt{\tilde \rho} u\|_{L^2}^2 + \|\rho\|_{L^2}^2 + \|b\|_{L^2}^2\Big) \\
    & + w_0 \|\nabla u\|_{L^2}^2  = \sum_{i = 1}^3 \widetilde{I}_i,
\end{split}
\end{equation}
where
$$ \widetilde{I}_1 = w_0 \;I_1, \qquad \widetilde{I}_2 = w_0\; I_2, \qquad \widetilde{I}_3= w_0 \;I_3. $$
Clearly $\widetilde{I}_1 =0$ since $I_1=0$. By the trivial fact that $|w_0| \leq 1$ and the previous bounds on $I_2$ and $I_3$,
\begin{equation}\nonumber
  \begin{split}
   |\widetilde{I}_2| \leq  |I_2| \lesssim& \; \|\nabla \cdot u\|_{L^\infty} \|\rho\|_{L^2}^2, \\
    |\widetilde{I}_3| \leq |I_3| \lesssim&  \; \|q(\rho)\|_{L^1} \|\nabla \cdot u\|_{L^\infty}  \lesssim \; \|\nabla \cdot u\|_{L^\infty} \|\rho\|_{L^2}^2.
  \end{split}
\end{equation}
Integrating \eqref{I1234} in time on $[0, t]$ yields
\begin{equation}\label{i1234}
\begin{split}
  \sup_{0\leq \tau \leq t} w_0 \big( \|u\|_{L^2}^2 + & \|\rho\|_{L^2}^2 + \|b\|_{L^2}\big) + \int_0^t w_{-1} \big( \|u\|_{L^2}^2 + \|\rho\|_{L^2}^2 + \|b\|_{L^2}\big) \; d\tau \\
  & + \int_0^t w_0 \|\nabla u\|_{L^2}^2 \; d\tau  \lesssim  \;\mathcal{E}_0 (0) + \mathcal{E}_0 (t) \big(\mathcal{E}_2^\frac{1}{2} (t) + \mathcal{F}_2^\frac{1}{2} (t)\big).
\end{split}
\end{equation}
Here we have used the approximation $\|\sqrt{\tilde \rho} u\|_{L^2}^2 \sim \|u\|_{L^2}^2$.  \eqref{i1234} is just the $L^2$ estimate of $\mathcal{E}_0(t)$.

\vskip .1in
Next we turn to the highest-order derivative parts in $\mathcal{E}_0(t)$ and $\mathfrak{E}(t)$.
Let $k=0,1$. Applying $\partial^{s-k}$ to the equations of $\rho$, $u$ and $b$ in \eqref{mhd1}, taking the $L^2$-inner product of the resulting equations with $(\partial^{s-k}\rho, \partial^{s-k}u, \partial^{s-k}b)$ and  multiplying by the time weight $w_0^{1-k}$, we obtain
\begin{equation}\label{I415}
\begin{split}
&\frac{1}{2} \frac{d}{dt} w_0^{1-k} \Big(\|\rho\|_{\dot H^{s-k}}^2 + \|u\|_{\dot H^{s-k}}^2 + \|b\|_{\dot H^{s-k}}^2 \Big) + w_0^{1-k}(t)\|\nabla u\|_{\dot H^{s-k}}^2 \\
&+\frac{(1-k)\sigma}{2}w_0^{-k} w_{-1} \Big(\|\rho\|_{\dot H^{s-k}}^2 + \|u\|_{\dot H^{s-k}}^2 + \|b\|_{\dot H^{s-k}}^2 \Big)   =  \sum_{i=4}^{15} I_i,
\end{split}
\end{equation}
where
\begin{align*}
I_4 = & \;  w_0^{1-k} \Big(- \int_{\mathbb{T}^2}  \partial^{s-k} \nabla \cdot u \partial^{s-k} \rho\; dx - \int_{\mathbb{T}^2} \partial^{s-k} \nabla \rho \cdot
\partial^{s-k} u\; dx \Big)\\
& +  w_0^{1-k} \Big( \int_{\mathbb{T}^2} \partial^{s-k} \nabla^{\bot} \cdot b \partial^{s-k} u_1 + \int_{\mathbb{T}^2} \partial^{s-k} \nabla^{\bot}  u_1\cdot \partial^{s-k} b \; dx\Big),\\
I_5 = & - w_0^{1-k}\int_{\mathbb{T}^2} \partial^{s-k} (u \cdot \nabla \rho) \partial^{s-k} \rho \; dx,
\\ I_6 =& - w_0^{1-k}\int_{\mathbb{T}^2} \partial^{s-k}(\rho \nabla \cdot u) \partial^{s-k} \rho \; dx, \\
I_7 = & - w_0^{1-k}\int_{\mathbb{T}^2} \partial^{s-k} (u \cdot \nabla u) \cdot\partial^{s-k} u \; dx,
\\ I_8 =&   - w_0^{1-k}\int_{\mathbb{T}^2} \partial^{s-k} \Big(\frac{\rho}{\rho + 1} \Delta u \Big)\cdot \partial^{s-k} u \; dx, \\
I_9 = & - w_0^{1-k}\int_{\mathbb{T}^2} \partial^{s-k} \Big(\frac{1}{\rho + 1} b \cdot \nabla b\Big) \cdot\partial^{s-k} u \; dx,\\
I_{10} =& - w_0^{1-k}\int_{\mathbb{T}^2} \partial^{s-k} \Big(\frac{\rho}{\rho + 1} \nabla^{\bot} \cdot b \Big) \partial^{s-k} u_1 \; dx, \\
I_{11} = & \; w_0^{1-k}\int_{\mathbb{T}^2} \partial^{s-k} \Big(\nabla \rho - \frac{1}{\rho + 1} \nabla P \Big) \cdot \partial^{s-k} u \; dx,\\
 I_{12} =&  - w_0^{1-k}\int_{\mathbb{T}^2} \partial^{s-k} (u \cdot \nabla b)\cdot \partial^{s-k} b\; dx, \\
I_{13} = & \; w_0^{1-k}\int_{\mathbb{T}^2} \partial^{s-k}(b \cdot \nabla u) \cdot\partial^{s-k} b \; dx,\\
 I_{14} =&  - w_0^{1-k}\int_{\mathbb{T}^2} \partial^{s-k}(b \nabla \cdot u) \cdot\partial^{s-k} b \; dx,  \\
I_{15} = & - w_0^{1-k}\int_{\mathbb{T}^2} \partial^{s-k} \Big(\frac{1}{2(\rho+1)} \nabla |b|^2 \Big) \cdot\partial^{s-k} u \; dx.
\end{align*}
By integration by parts,
\begin{equation} \label{I4}
I_4  = 0.
\end{equation}
To efficiently estimate the terms $I_5$ through $I_{14}$ and avoid repetitions, we divide these
terms into three types: transport terms, terms containing $\na\cdot u$ and other terms.

\quad \\
\noindent
{\bf Type I: Transport terms containing $u\cdot \nabla$ or $b \cdot \nabla $ }
~\\~

This part contains $I_5, I_7, I_9, I_{12}$ and $I_{13}$. We shall first handle the terms containing $u\cdot \nabla$ operator. By Lemma \ref{prop1},  for $g = u, b, \rho$,
\begin{align*}
&\int_0^t |I_5| + |I_7| + |I_{12}| \; d \tau \\
 \lesssim & \; \sup_{0\leq \tau \leq t} w_0^{1-k} \|g\|_{ \dot H^{s-k}}^2 \int_0^t
\|\nabla u_1\|_{H^{[\frac{s-k}{2}] + 1}} \; d \tau \\
&+ \sup_{0\leq \tau \leq t} w_0^\frac{1-k}{2} \|g\|_{\dot H^{s-k}} \int_0^t \|\nabla u_2\|_{H^{[\frac{s-k}{2}] + 1}} w_0^\frac{1-k}{2} \|\partial_2g\|_{H^{s-k-1}} \; d\tau  \\
&+ \sup_{0\leq \tau \leq t} w_0^{1-k} \|g\|_{H^{s-k}}^2 \int_0^t  \big( \|\nabla u_1\|_{H^{s-k-1}} + \| \nabla \cdot u \|_{H^2}   \big)   \; d\tau  \\
&+  \sup_{0\leq \tau \leq t} w_0^{1-k} \|u\|_{H^{s-k}} \|g\|_{H^{s-k}} \int_0^t \| \partial_2 g\|_{H^{[\frac{s-k}{2}] + 2}} \; d\tau .
\end{align*}
Notice that the average integral of $\partial_2 g $ equals zero, the Poinc\'{a}re's inequality directly implies $\| \partial_2 g\|_{H^{[\frac{s-k}{2}] + 2}} \leq  \| \partial_2^2 g\|_{H^{s-3}}$ since $ s \ge 10$.
For the case $k = 0$, we have the following estimate
\begin{align}
& \int_0^t  |I_5| + |I_7| + |I_{12}|\; d \tau \notag\\
\lesssim & \; \sup_{0\leq \tau \leq t} w_0   \|g\|_{\dot H^{s}}^2 \int_0^t
\|\partial_2 u_1\|_{H^{s-1}} + \|\partial_1 u_1\|_{H^{s-1}}\; d \tau \notag\\
&+ \sup_{0\leq \tau \leq t} w_0^\frac{1}{2} \|g\|_{\dot H^s} \int_0^t \|\nabla u_2\|_{H^{s-1}} w_{-1}^\frac{1}{4} \|g\|_{H^s} ^\frac{1}{2} w_1^\frac{1}{4} \|\partial_2^2 g\|_{H^{s-2}}^\frac{1}{2}\; d\tau  \notag\\
&+ \sup_{0\leq \tau \leq t} w_0 \|g\|_{H^s}^2 \int_0^t  \big( \|\partial_1 u_1\|_{H^{s-1}} + \| \partial_2 u \|_{H^{s-1}}   \big)   \; d\tau  \notag\\
&+  \sup_{0\leq \tau \leq t} w_0 \|u\|_{H^s} \|g\|_{H^s} \int_0^t \| \partial_2^2 g\|_{H^{s-3}} \; d\tau \notag\\
\lesssim &  \; \mathcal{E}_0 (t) \big(\mathcal{E}_2^\frac{1}{2} (t)+ \mathcal{F}_2^\frac{1}{2}(t)\big)
 + \; \mathfrak{E}^\frac{1}{2} (t)\mathcal{E}_0^\frac{1}{2}(t)\big( \mathcal{E}_0^\frac{1}{4} (t) \mathcal{E}_1^\frac{1}{4} (t)+\mathcal{E}_0^\frac{1}{4}(t)\mathcal{B}_1^\frac{1}{4}(t)+ \mathcal{E}_0^\frac{1}{4}(t)\mathcal{P}_1^\frac{1}{4}(t)      \big) \notag\\
& + \mathcal{E}_0 (t)\big( \mathcal{E}_2^\frac{1}{2} (t)+ \mathcal{P}_2^\frac{1}{2} (t)+  \mathcal{B}_2^\frac{1}{2} (t)\big).\label{0I5712}
\end{align}
Where we have used the facts: $w_0 = w_{-1}^\frac{1}{2}w_1^\frac{1}{2}$ and $ \|\partial_2 g\|_{H^{s-1}} \lesssim \|\partial_2^2 g\|_{H^{s-2}}^\frac{1}{2} \|g\|_{H^s}^\frac{1}{2}$.
For the case $k = 1$, similarly, the time integral is bounded by
\begin{equation}\label{1I5712}
\begin{split}
& \int_0^t |I_5| + |I_7| + |I_{12}|  \; d \tau \\
\lesssim & \; \sup_{0\leq \tau \leq t} \| g\|_{\dot H^{s-1}}^2 \int_0^t \| \nabla u_1 \|_{H^{[\frac{s-1}{2}]+1}} \; d \tau  \\
& + \sup_{0\leq \tau \leq t} \| g\|_{\dot H^{s-1}} \int_0^t \| \nabla u_2 \|_{H^{[\frac{s-1}{2}]+1}}  \| \partial_2 g\|_{H^{s-2}} \; d \tau \\
& + \sup_{0\leq \tau \leq t} \| g\|_{H^{s-1}}^2 \int_0^t \big( \| \nabla u_1 \|_{H^{s-2}} + \| \nabla\cdot u\|_{H^2} \big)\; d \tau \\
&+ \sup_{0\leq \tau \leq t}\| u\|_{H^{s-1}} \|g\|_{H^{s-1}} \int_0^t \| \partial_2 g \|_{H^{[\frac{s-1}{2}]+2}} \; d \tau  \\
\lesssim & \;  \mathfrak{E} (t)\big( \mathcal{E}_2^\frac{1}{2}(t)+ \mathcal{F}_2^\frac{1}{2}(t)  \big) +\mathfrak{E}^\frac{1}{2} (t)\mathcal{E}_0^\frac{1}{2}(t)\big( \mathcal{E}_0^\frac{1}{4} (t) \mathcal{E}_1^\frac{1}{4} (t)+\mathcal{E}_0^\frac{1}{4}(t)\mathcal{B}_1^\frac{1}{4}(t)+ \mathcal{E}_0^\frac{1}{4}(t)\mathcal{P}_1^\frac{1}{4}(t)      \big) \\
 & + \mathfrak{E}(t)\big(\mathcal{E}_2^\frac{1}{2}(t)+ \mathcal{P}_2^\frac{1}{2} (t)+ \mathcal{B}_2^\frac{1}{2}(t)\big).
\end{split}
\end{equation}
The same strategy can be used for the terms containing $ b\cdot \nabla$ operator such as $I_9$. By direct calculations we can derive,
\begin{equation}\nonumber
\begin{split}
| I_9 | \lesssim & \; w_0^{1-k}\Big( \Big\|\frac{1}{\rho+1} \Big\|_{L^\infty} \big\|b\cdot \nabla b \big\|_{H^{s-k-1}} + \Big\|\frac{1}{\rho+1}\Big\|_{H^{s-k-1}} \big\|b \cdot \nabla b\big\|_{L^\infty}\Big) \|\nabla u\|_{H^{s-k}} \\
\lesssim & \; w_0^{1-k}\big(\|b_1\|_{L^\infty} \|\partial_1 b\|_{H^{s-k-1}} + \|b_1\|_{H^{s-k-1}} \|\partial_1 b\|_{L^\infty}\big)\|\nabla u\|_{H^{s-k}} \\
& +
w_0^{1-k}\big(\|b_2\|_{L^\infty} \|\partial_2 b\|_{H^{s-k-1}} + \|b_2\|_{H^{s-k-1}} \|\partial_2 b\|_{L^\infty}\big) \|\nabla u\|_{H^{s-k}} \\
& + w_0^{1-k}\big(1 + \|\rho\|_{H^{s-k-1}}\big) \big(\|b_1\|_{L^\infty} \|\partial_1 b\|_{L^\infty} + \|b_2\|_{L^\infty} \|\partial_2 b\|_{L^\infty}\big) \|\nabla u\|_{H^{s-k}}.
\end{split}
\end{equation}
Because $\int_{\mathbb{T}^2} b_0(x) \; dx = 0$, we have
$$
\int_{\mathbb{T}^2} b(x,t) \; dx = 0
$$
and thus we can apply Poinc\'{a}re's inequality on the periodic domain $\mathbb{T}^2$ to obtain
$$ \|b_1\|_{H^{s_0}} \lesssim \|\partial_2 b\|_{H^{s_0 - 1}},\quad  s_0 \in \mathbb{N}^+ .$$
When $k = 0$, we still invoke the interpolation inequalities
$$
 w_0(t) = w_{-1}^{\frac{1}{2}}(t) w_{1}^{\frac{1}{2}}(t), \qquad \|\partial_2 b\|_{H^{s-1}} \lesssim \|b\|_{H^s}^\frac{1}{2} \|\partial_2^2 b\|_{H^{s-2}}^\frac{1}{2}
$$
to obtain
\begin{equation}\label{0I9}
  \begin{split}
     \int_0^t   |I_9| \; d\tau
    \lesssim & \sup_{0\leq \tau \leq t} w_0^\frac{1}{2} \|b\|_{H^{s}} \int_0^t w_0^\frac{1}{2} \|\nabla u\|_{H^s} \|b_1\|_{L^\infty} \; dt \\
    & + \sup_{0\leq \tau \leq t}    \big(1+\|\rho\|_{H^{s-1}} \big)\| b\|_{H^{s-1}}\\
    &\qquad \times  \int_0^t w_{-1}^\frac{1}{4}\|b\|^\frac{1}{2}_{H^{s}}  w_1^\frac{1}{4} \|\partial_2^2 b\|^\frac{1}{2}_{H^{s-2}}  w_{0}^\frac{1}{2}  \| \nabla u\|_{H^s}  \; d\tau \\
    \lesssim&  \; \mathcal{E}_0(t) \mathcal{B}_2^\frac{1}{2}(t) + \mathfrak{E}^\frac{1}{2}(t) \mathcal{B}_1^\frac{1}{4} (t)\mathcal{E}_0^\frac{3}{4}(t).
  \end{split}
\end{equation}
Here we have used \eqref{assum}.
Similarly, for $k = 1$, the time integral of $|I_9|$ is bounded by
\begin{equation}\label{1I9}
  \begin{split}
     \int_0^t |I_9| \; d\tau
    \lesssim   \sup_{0\leq \tau \leq t} \big(1+\|\rho\|_{H^{s-2}}\big)  \|b\|_{H^{s-1}} & \int_0^t w_{2}^\frac{1}{4}\|\partial_2^2 b\|^\frac{1}{2}_{H^{s-4}} w_{-1}^\frac{1}{4}\|b\|^\frac{1}{2}_{H^{s}} \| \nabla u\|_{H^{s-1}} \; d\tau \\
    & \lesssim   \; \mathfrak{E}(t)  \mathcal{B}_2^\frac{1}{4} (t)\mathcal{E}_0^\frac{1}{4}(t).
  \end{split}
\end{equation}
Here we have used the facts that $ w_{-1}(\tau)w_2(\tau)= (1+\tau)^{ 1-2\sigma }>1$ for $0< \sigma <\frac{1}{2}$ and $\|\partial_2 b\|_{H^{s-2}} \lesssim \|\partial_2^2 b\|_{H^{s-4}}^\frac{1}{2} \|b\|_{H^s}^\frac{1}{2}$.

\vskip .1in
The next term $I_{13}$ also contains operator $b\cdot\nabla $. Following the estimates in Lemma \ref{prop1}, we have
\begin{align*}
|I_{13}|
\lesssim & \; w_0^{1-k}
\big(\|\nabla b_1\|_{H^{[\frac{s-k}{2}] + 1}} \|\partial_1 u\|_{H^{s-k-1}} +\|\nabla b_2\|_{H^{[\frac{s-k}{2}] + 1}} \|\partial_2 u\|_{H^{s-k-1}} \big)\|b\|_{\dot H^{s-k}} \\
&+w_0^{1-k} \big(\|\nabla b_1\|_{H^{s-k-1}} \| \partial_1 u\|_{H^{[\frac{s-k}{2}] + 2}}+ \|\nabla b_2\|_{H^{s-k-1}} \| \partial_2 u\|_{H^{[\frac{s-k}{2}] + 2}}\big) \|b\|_{\dot H^{s-k}} \\
& + w_0^{1-k}  \big(\|b_1\|_{L^\infty}  \|\partial_1 u\|_{\dot H^{s-k}}  + \|b_2\|_{L^\infty} \|\partial_2 u\|_{\dot H^{s-k}}\big)\|b\|_{\dot H^{s-k}}.
\end{align*}
Like before, for the case $k = 0$,  we can bound its time integral as follows,
\begin{align}
\int_0^t |I_{13}| \; d \tau  \lesssim& \sup_{0\leq \tau \leq t} w_0 \|u\|_{ H^s}\|b\|_{ H^{s}} \int_0^t
\|\nabla b_1\|_{H^{[\frac{s}{2}] + 1}} \; d \tau \notag \\
&+ \sup_{0\leq \tau \leq t} w_0 \|b\|_{ H^{s}}^2 \int_0^t \|\partial_2 u\|_{H^{s-1}} \; d\tau \notag\\
&+ \sup_{0\leq \tau \leq t} w_0^\frac{1}{2} \|b\|_{H^s} \int_0^t w_0^\frac{1}{2}\|\nabla b_1\|_{H^{s-1}} \| \partial_1 u\|_{H^{[\frac{s}{2}] + 2}} \; d\tau\notag \\
&+  \sup_{0\leq \tau \leq t} w_0  \|b\|_{H^s}^2  \int_0^t \| \partial_2 u\|_{H^{[\frac{s}{2}] + 2}} \; d\tau\notag\\
&+ \sup_{0\leq \tau \leq t} w_0^\frac{1}{2} \|b\|_{H^s} \int_0^t w_0^\frac{1}{2} \|\nabla u\|_{H^s} \|b_1\|_{L^\infty} \; d\tau\notag\\
& + \sup_{0\leq \tau \leq t} \|b_2\|_{L^\infty} \int_0^t w_1^\frac{1}{2} \|\partial_2 u\|_{\dot H^s} w_{-1}^\frac{1}{2} \|b\|_{\dot H^s} \; d\tau \notag\\
\lesssim& \; \mathcal{E}_0 (t)\mathcal{B}_2^\frac{1}{2} (t)+ \mathcal{E}_0(t) \mathcal{E}_2^\frac{1}{2}(t) + \mathfrak{E}^\frac{1}{2}(t) \mathcal{B}_1^\frac{1}{4} (t) \mathcal{E}_0^\frac{3}{4}(t)\notag\\
& + \mathcal{E}_0(t) \mathcal{B}_1^\frac{1}{2}(t)+\mathfrak{E}^\frac{1}{2}(t) \mathcal{E}_1^\frac{1}{2}(t) \mathcal{E}_0^\frac{1}{2}(t).\label{0I13}
\end{align}
Similarly, for $k = 1$, we have
\begin{align}
\int_0^t |I_{13}| \; d \tau  \lesssim& \sup_{0\leq \tau \leq t}  \|u\|_{ H^{s-1}}\|b\|_{ H^{s-1}} \int_0^t
\|\nabla b_1\|_{H^{[\frac{s-1}{2}] + 1}} \; d \tau \notag \\
&+ \sup_{0\leq \tau \leq t}  \|b\|_{ H^{s-1}}^2  \int_0^t \|\partial_2 u\|_{H^{s-2}} \; d\tau \notag\\
&+ \sup_{0\leq \tau \leq t} \|b\|_{H^{s-1}} \int_0^t  \|\nabla b_1\|_{H^{s-2}} \| \partial_1 u\|_{H^{[\frac{s-1}{2}] + 2}} \; d\tau\notag \\
&+  \sup_{0\leq \tau \leq t} \|b\|_{H^{s-1}}^2  \int_0^t  \| \partial_2 u\|_{H^{[\frac{s-1}{2}] + 2}} \; d\tau\notag\\
&+ \sup_{0\leq \tau \leq t}  \|b\|_{H^{s-1}} \int_0^t  \|\nabla u\|_{H^{s-1}} \|b_1\|_{L^\infty} \; d\tau\notag\\
& + \sup_{0\leq \tau \leq t} \|b_2\|_{L^\infty} \|b\|_{\dot H^{s-1}} \int_0^t   \|\partial_2 u\|_{\dot H^{s-1}}   \; d\tau \notag\\
\lesssim& \; \mathfrak{E}(t) \mathcal{B}_2^\frac{1}{2} (t)+ \mathfrak{E} (t)\mathcal{E}_2^\frac{1}{2}(t) + \mathfrak{E} (t) \mathcal{B}_1^\frac{1}{2}(t)+
\mathfrak{E}^\frac{1}{2} (t) \mathcal{E}_1^\frac{1}{2}(t)  \mathcal{B}_1^\frac{1}{2}(t).    \label{1I13}
\end{align}
We have finished the estimates for all terms containing $u\cdot\nabla$ or $b\cdot\nabla$ by now.

~\\~
\noindent
{\bf Type II: Terms containing the divergence  of velocity $\nabla \cdot u$}
~\\~

This part contains three terms $I_6, I_{11}$ and $I_{14}$.
$I_{11}$ will contain the divergence of velocity $\nabla\cdot u$ after integration by parts.
Since $I_6$ and $I_{14}$ have similar structures, they can be treated simultaneously. By H\"{o}lder's inequality,
\begin{equation}\nonumber
  \begin{split}
    |I_6| + |I_{14}| \lesssim & \;w_0^{1-k} \Big( \big\|(b| \rho)\big\|_{L^\infty} \|\nabla \cdot u\|_{H^{s-k}} + \big\|(b| \rho)\big\|_{H^{s-k}} \|\nabla \cdot u\|_{L^\infty} \Big) \big\|(b| \rho) \big\|_{H^{s-k}}.
  \end{split}
\end{equation}
Here we have used the notation $ (b|\rho) $ to stand for $b$ and $\rho$.
For $k = 0$,  we have
\begin{equation}\label{0I614}
\begin{split}
 \int_0^t |I_6| + |I_{14}| \; d\tau \lesssim & \sup_{0\leq \tau \leq t}  w_0  \big\|(b|\rho)\big\|_{H^s}^2 \int_0^t \|\nabla \cdot u\|_{H^2} \; d\tau \\
& +\sup_{0\leq \tau \leq t}  \big\|(b | \rho)\big\|_{H^2} \int_0^t w_1^\frac{1}{2} \|\nabla \cdot u\|_{H^s}w_{-1}^\frac{1}{2} \big\|(b| \rho)\big\|_{H^s} \; d\tau \\
 \lesssim  & \; \mathcal{E}_0(t) \big(\mathcal{E}_2^\frac{1}{2} (t)+ \mathcal{F}_2^\frac{1}{2}(t)\big)+ \mathfrak{E}^\frac{1}{2} (t)\big(\mathcal{E}_1^\frac{1}{2}(t) + \mathcal{F}_1^\frac{1}{2}(t)\big) \mathcal{E}_0^\frac{1}{2}(t)
.
\end{split}
\end{equation}
For $k = 1$,
\begin{equation}\label{1I614}
\begin{split}
 \int_0^t |I_6| + |I_{14}| \; d\tau \lesssim & \sup_{0\leq \tau \leq t}  \big\|(b|\rho)\big\|_{H^{s-1}}^2 \int_0^t \|\nabla \cdot u\|_{H^2} \; d\tau \\
& +\sup_{0\leq \tau \leq t}  \big\|(b | \rho)\big\|_{H^2} \big\|(b| \rho)\big\|_{H^{s-1}}  \int_0^t  \|\nabla \cdot u\|_{H^{s-1}}   \; d\tau \\
 \lesssim  & \; \mathfrak{E} (t)\big(\mathcal{E}_2^\frac{1}{2}(t) + \mathcal{F}_2^\frac{1}{2} (t)\big).
\end{split}
\end{equation}
$I_{11}$ can be written as
\begin{equation}\nonumber
\begin{split}
I_{11} = & - w_0^{1-k}\int_{\mathbb{T}^2} \partial^{s-k} \Big( \Big(\frac{P'(\tilde \rho)}{\rho+1} - 1 \Big) \nabla \rho \Big) \cdot \partial^{s-k} u \; dx\\
= & - w_0^{1-k}\int_{\mathbb{T}^2} \partial^{s-k} \nabla q_1(\rho) \cdot \partial^{s-k} u \; dx \\
= & \;  w_0^{1-k}\int_{\mathbb{T}^2} \partial^{s-k} q_1(\rho)  \partial^{s-k} \nabla \cdot u \; dx,
\end{split}
\end{equation}
where
$$
q_1(\rho) = \int_0^\rho \left(\frac{P'(r+1)}{r+1} - 1 \right)d r .
$$
Clearly, because $P'(1)=1$, $q_1(0) = q_1'(0)=0$  and we shall derive  $q_1(\rho) \sim \rho^2$. Therefore,
$$
|I_{11}| \lesssim w_0^{1-k}\|q(\rho)\|_{H^{s-k}} \|\nabla \cdot u\|_{H^{s-k}}
$$
and, for $k=0$,
\begin{equation}\label{0I11}
  \begin{split}
    \int_0^t |I_{11}| \; d\tau \lesssim   & \; \sup_{0\leq \tau \leq t}  \|\rho\|_{H^{s-1}} \int_0^t w_{-1}^\frac{1}{2} \|\rho\|_{H^s} w_1^\frac{1}{2}\|\nabla \cdot u\|_{H^s} \; d\tau \\
    \lesssim  & \; \mathfrak{E}^\frac{1}{2}(t) \mathcal{E}_0^\frac{1}{2}(t)\big(\mathcal{E}_1^\frac{1}{2}(t) + \mathcal{F}_1^\frac{1}{2}(t)\big).
  \end{split}
\end{equation}
For $k = 1$,
\begin{equation}\label{1I11}
  \begin{split}
    \int_0^t |I_{11}| \; d\tau \lesssim  \sup_{0\leq \tau \leq t}  \|\rho\|_{H^{s-1}}^2 \int_0^t \|\nabla \cdot u\|_{H^{s-1}} \; d\tau
    \lesssim  \; \mathfrak{E}(t)(\mathcal{E}_2^\frac{1}{2}(t) + \mathcal{F}_2^\frac{1}{2}(t)).
  \end{split}
\end{equation}
This completes the estimates for $|I_{11}|$ and also for all terms of Type II.

\quad \\
\noindent
{\bf Type III: Basic good terms}
~\\~

The remaining terms $I_8, I_{10}, I_{15}$ are referred to as basic good terms and are
handled in this part. First, $I_8$ can be bounded as follows.
\begin{equation}\nonumber
\begin{split}
  |I_8| =  &\;w_0^{1-k} \Big|\int_{\mathbb{T}^2} \partial^{s-k -1} \Big(\frac{\rho}{\rho + 1} \Delta u \Big)\cdot \partial^{s-k +1} u \; dx \Big| \\
  \lesssim & \;w_0^{1-k} \big( \Big\|\frac{\rho}{\rho + 1} \Big\|_{L^\infty} \|\Delta u\|_{H^{s-k -1}} + \Big\|\frac{\rho}{\rho + 1}\Big\|_{H^{s-k-1}} \|\Delta u\|_{L^\infty} \big) \|\nabla u\|_{H^{s-k}} \\
  \lesssim & \;\|\rho\|_{H^{s-k - 1}} w_0^{1-k} \|\nabla u\|_{H^{s-k}}^2 + \|\rho\|_{H^{s-k}} w_0^\frac{1-k}{2}  \|\nabla u\|_{H^{s-k-1}} w_0^\frac{1-k}{2} \|\nabla u\|_{H^{s-k}}.
\end{split}
\end{equation}
For $k = 0$,
\begin{equation}\label{0I8}
  \int_0^t |I_8| \;d\tau \lesssim \mathcal{E}_0(t) \mathfrak{E}^\frac{1}{2}(t)
\end{equation}
and for $k = 1$,
\begin{equation}\label{1I8}
  \int_0^t |I_8| \;d\tau \lesssim \mathfrak{E}^\frac{3}{2}(t).
\end{equation}
Similarly,
\begin{equation}\nonumber
\begin{split}
|I_{10}| = & \; w_0^{1-k} \Big| \int_{\mathbb{T}^2} \partial^{s-k-1} \Big(\frac{\rho}{\rho + 1} \nabla^{\bot} \cdot b \Big) \partial^{s-k+1} u_1 \; dx \Big| \\
\lesssim & \; w_0^{1-k}\|\rho\|_{H^{s-1}} \|b\|_{H^{s-k}}  \|\nabla  u_1\|_{H^{s-k}}.
\end{split}
\end{equation}
Therefore, for $k = 0$, the time integral of $|I_{10}|$ is bounded by
\begin{equation}\label{0I10}
\begin{split}
\int_0^t |I_{10}| \; d\tau \lesssim  & \;\sup_{0\leq \tau \leq t} \|\rho\|_{H^{s-1}} \int_0^t w_{-1}^\frac{1}{2}\|b\|_{H^s} w_1^\frac{1}{2}
\|\nabla  u_1\|_{H^{s}} \; d\tau \\
\lesssim  & \; \mathfrak{E}^\frac{1}{2}(t) \mathcal{E}_0^\frac{1}{2}(t) \big(\mathcal{E}_1^\frac{1}{2}(t) + \mathcal{F}_1^\frac{1}{2} (t)\big).
\end{split}
\end{equation}
For $k = 1$,
\begin{equation}\label{1I10}
\begin{split}
\int_0^t |I_{10}| \; d\tau \lesssim  \;\sup_{0\leq \tau \leq t} \|\rho\|_{H^{s-1}} \|b\|_{H^{s-1}}\int_0^t
\|\nabla  u_1\|_{H^{s-1}} \; d\tau
\lesssim \; \mathfrak{E}(t)  \big(\mathcal{E}_2^\frac{1}{2}(t) + \mathcal{F}_2^\frac{1}{2}(t)\big).
\end{split}
\end{equation}

Finally we deal with $I_{15}$.  To distinguish the horizontal derivative from the vertical one,
we write
\begin{align*}
\partial^{s-k-1}\Big(\frac{1}{2(\rho+1)} \nabla |b|^2 \Big) \cdot\partial^{s-k+1} u
& =\partial^{s-k-1}\Big(\frac{1}{(\rho+1)} b\cdot \p_1 b \Big) \cdot\partial^{s-k+1} u_1\\
&\quad + \partial^{s-k-1}\Big(\frac{1}{(\rho+1)} b\cdot \p_2 b \Big) \cdot\partial^{s-k+1} u_2.
\end{align*}
By Sobolev's inequality,
\begin{equation}\nonumber
\begin{split}
|I_{15}|
=  & \; w_0^{1-k} \Big| \int_{\mathbb{T}^2} \partial^{s-k-1}\Big(\frac{1}{2(\rho+1)} \nabla |b|^2 \Big) \cdot\partial^{s-k+1} u \; dx   \Big|  \\
\lesssim & \; w_0^{1-k}(1+ \|\rho\|_{H^{s-k-1}})\,\big(\|b\|_{H^{s-k-2}} \|b\|_{H^{s-k}})\\
&\qquad\times \big(\|\nabla \partial_1 u_1\|_{H^{s-k-1}} +
\|\nabla \partial_2 u_1\|_{H^{s-k-1}}\big) \\
& + w_0^{1-k} (1+ \|\rho\|_{H^{s-k-1}})\,\|\nabla u\|_{H^{s-k}}\\
&\qquad \times \big(\|\partial_2 b\|_{H^{s-k-1}} \|b\|_{H^{s-k-2}} + \|b\|_{H^{s-k-1}} \|\partial_2 b\|_{H^{s-k-2}}\big).
\end{split}
\end{equation}
For $k = 0$, we can bound the time integral by
\begin{equation}\label{0I15}
  \begin{split}
    \int_0^t |I_{15}| \; d\tau \lesssim & \sup_{0\leq \tau \leq t} \|b\|_{H^{s-2}} (1+ \|\rho\|_{H^{s-1}})\\
    &\quad \times \int_0^t w_{-1}^\frac{1}{2} \|b\|_{H^{s}}w_1^\frac{1}{2}\big(\|\nabla \partial_1 u_1\|_{H^{s-1}} +
\|\nabla \partial_2 u_1\|_{H^{s-1}}\big) \; d\tau \\
& + \sup_{0\leq \tau \leq t} \|b\|_{H^{s-1}} (1+ \|\rho\|_{H^{s-1}})\int_0^t w_0^\frac{1}{2} \|\partial_2 b\|_{H^{s-1}} w_0^\frac{1}{2} \|\nabla u\|_{H^{s}} \; d\tau \\
\lesssim & \; \mathfrak{E}^\frac{1}{2}(t) \mathcal{E}_0^\frac{1}{2}(t) \big(\mathcal{E}_1^\frac{1}{2}(t) + \mathcal{F}_1^\frac{1}{2}(t)\big) + \mathfrak{E}^\frac{1}{2}(t)\mathcal{E}_0^\frac{3}{4}(t)\mathcal{B}_1^\frac{1}{4}(t).
  \end{split}
\end{equation}
For $k = 1$,
\begin{equation}\label{1I15}
  \begin{split}
    \int_0^t |I_{15}|\; d\tau \lesssim & \sup_{0\leq \tau \leq t}   \|b\|_{H^{s-1}}^2 (1+ \|\rho\|_{H^{s-2}})\\
    &\quad \times  \int_0^t \big(\|\nabla \partial_1 u_1\|_{H^{s-2}} +
\|\nabla \partial_2 u_1\|_{H^{s-2}}\big) \; d\tau \\
& + \sup_{0\leq \tau \leq t}  \|b\|_{H^{s-2}} (1+ \|\rho\|_{H^{s-2}}) \int_0^t \|\partial_2 b\|_{H^{s-2}} \|\nabla u\|_{H^{s-1}} \; d\tau \\
\lesssim & \; \mathfrak{E}(t)   \big(\mathcal{E}_2^\frac{1}{2}(t) + \mathcal{F}_2^\frac{1}{2}(t)  \big) + \mathfrak{E}(t) \mathcal{B}_1^\frac{1}{2}(t).
  \end{split}
\end{equation}
This completes the estimates for all terms of Type III.

\vskip .1in
In the case when $k=0$, we integrate \eqref{I415} in time and take into account of the upper bounds \eqref{I4}, \eqref{0I5712}, \eqref{0I9}, \eqref{0I13}, \eqref{0I614}, \eqref{0I11}, \eqref{0I8},  \eqref{0I10} and \eqref{0I15} to obtain the desired bound for the highest-order energy in $\mathcal{E}_0(t)$. Combining it with the $L^2$ estimate \eqref{i1234},  we conclude that
$$ \mathcal{E}_0(t)  \lesssim \; \mathcal E_0(0) +  \mathcal{E}_{total}^{\frac{3}{2}}(t). $$
Here we have bounded the intermediate energies in terms of $\mathcal{E}_{total}$ through interpolation.

\vskip .1in
Similarly, when $k=1$, we integrate \eqref{I415} in time and invoke the upper bounds in \eqref{I4}, \eqref{1I5712}, \eqref{1I9}, \eqref{1I13}, \eqref{1I614}, \eqref{1I11}, \eqref{1I8},  \eqref{1I10} and \eqref{1I15} to achieve the bound for the highest order energy part in $\mathfrak{E}(t)$.
Combining it with the $L^2$ estimate \eqref{i123} leads to
$$ \mathfrak{E}(t)  \lesssim \; \mathcal E_0(0) +  \mathcal{E}_{total}^{\frac{3}{2}}(t). $$
This completes the proof of Proposition \ref{lemmaBasic}.
\end{proof}

\vskip .3in
\section{Strong dissipation in vertical derivative}
\label{Strong}

This section presents the upper bounds for $\mathcal{E}_k(t)$,  $\mathcal{P}_k(t)$ and  $\mathcal{B}_k(t)$, the functionals involving the vertical derivatives. There estimates are
stated in the following three propositions.  The first one focuses on $\mathcal{E}_1(t)$ and
 $\mathcal{E}_3(t)$.

\begin{proposition}\label{strongDis}
	Let $\mathcal{E}_{total}(t)$ be the total energy defined in \eqref{total}. Then, for any $t>0$, 	
\begin{equation*}
\begin{split}
\mathcal{E}_1(t) \lesssim \; &  \mathcal{E}_1(0) + \mathcal{E}_0^{\frac{1}{2}} (t) \big( \mathcal{E}_1^{\frac{1}{2}}(t) + \mathcal{P}_1^{\frac{1}{2}}(t) +\mathcal{B}_1^{\frac{1}{2}}(t)  \big)+ \mathcal{E}^{\frac{3}{2}}_{total}(t), \\
\mathcal{E}_3(t) \lesssim \;&  \mathcal{E}_3(0) +  \big( \mathcal{E}_1^{\frac{1}{2}}(t) + \mathcal{P}_1^{\frac{1}{2}} (t)+\mathcal{B}_1^{\frac{1}{2}} (t) \big)\big( \mathcal{E}_3^{\frac{1}{2}}(t) + \mathcal{P}_3^{\frac{1}{2}}(t) +\mathcal{B}_3^{\frac{1}{2}} (t) \big) + \mathcal{E}^{\frac{3}{2}}_{total}(t).
\end{split}
\end{equation*}
\end{proposition}

\vskip .1in

Next we bound $\mathcal{P}_k(t)$.

\begin{proposition}\label{strongDis2}
		Let $\mathcal{E}_{total}(t)$ be the total energy defined in \eqref{total}. Then, for any $t>0$,
	\begin{equation*}
		\mathcal{P}_k(t) \lesssim \; \mathcal{E}_k(0) + \mathcal{E}_k(t)+\mathcal{E}^{\frac{3}{2}}_{total}(t),\qquad k=1, 3.
	\end{equation*}
\end{proposition}

The last proposition bounds $\mathcal{B}_k(t)$, the energy functional involving the vertical derivative
of the magnetic field.

\begin{proposition}\label{strongDis3}
	Let $\mathcal{E}_{total}(t)$ be the total energy defined in \eqref{total}. Then, for any $t>0$,
	\begin{equation*}
		\mathcal{B}_k(t) \lesssim \;\mathcal{E}_k(0)+ \mathcal{E}_k(t)+\mathcal{P}_k(t)+\mathcal{E}^{\frac{3}{2}}_{total}(t),\qquad k=1, 3.
	\end{equation*}
\end{proposition}

\vskip .1in
The rest of this section proves these propositions. We start with the first proposition.

\begin{proof}[Proof of Proposition \ref{strongDis}]
Due to Poinc\'{a}re's inequality, we only focus on the terms with the highest-order derivatives. According to \eqref{mhd1}, for $k=1,3$,
\begin{equation}\label{N}
\begin{split}
&\frac{1}{2} \frac{d}{dt} w_k \big(\|\partial_2 \rho\|_{\dot H^{s-k}}^2 + \|\partial_2 u\|_{\dot H^{s-k}}^2 + \|\partial_2 b\|_{\dot H^{s-k}}^2 \big)
+ w_k\|\nabla \partial_2 u\|_{\dot H^{s-k}}^2 = \sum_{i=1}^{13} N_i,
\end{split}
\end{equation}
where
\begin{align*}
N_1 = & \;\frac{k- \sigma}{2} w_{k-1} \big(\|\partial_2 \rho\|_{\dot H^{s-k}}^2 + \|\partial_2 u\|_{\dot H^{s-k}}^2 + \|\partial_2 b\|_{\dot H^{s-k}}^2\big), \\
N_2 = & \; w_k \Big(\int_{\mathbb{T}^2} - \partial^{s-k}\partial_2 \nabla \cdot u \partial^{s-k}\partial_2 \rho\; dx - \int_{\mathbb{T}^2} \partial^{s-k} \partial_2 \nabla \rho
\cdot\partial^{s-k} \partial_2 u\; dx \Big)\\
& +  w_k \Big( \int_{\mathbb{T}^2} \partial^{s-k}\partial_2\nabla^{\bot} \cdot b \partial^{s-k}\partial_2 u_1 + \int_{\mathbb{T}^2} \partial^{s-k}\partial_2 \nabla^{\bot}  u_1 \cdot\partial^{s-k}\partial_2 b \; dx \Big),\\
N_3 = & - w_k\int_{\mathbb{T}^2} \partial^{s-k}\partial_2 (u \cdot \nabla \rho) \partial^{s-k}\partial_2 \rho \; dx, \\
N_4 =& - w_k\int_{\mathbb{T}^2} \partial^{s-k}\partial_2(\rho \nabla \cdot u) \partial^{s-k}\partial_2 \rho \; dx, \\
N_5 = & - w_k\int_{\mathbb{T}^2} \partial^{s-k}\partial_2 (u \cdot \nabla u) \cdot\partial^{s-k}\partial_2 u \; dx,\\
N_6 =& - w_k\int_{\mathbb{T}^2} \partial^{s-k}\partial_2 \Big(\frac{\rho}{\rho + 1} \Delta u\Big) \cdot \partial^{s-k}\partial_2 u \; dx, \\
N_7 = & - w_k\int_{\mathbb{T}^2} \partial^{s-k}\partial_2 \Big(\frac{1}{\rho + 1} b \cdot \nabla b\Big) \cdot\partial^{s-k}\partial_2 u \; dx, \\
N_8 = & - w_k\int_{\mathbb{T}^2} \partial^{s-k}\partial_2 \Big(\frac{\rho}{\rho + 1} \nabla^{\bot} \cdot b\Big) \partial^{s-k}\partial_2 u_1 \; dx, \\
N_9 = & \; w_k\int_{\mathbb{T}^2} \partial^{s-k}\partial_2 \Big(\nabla \rho - \frac{1}{\rho + 1} \nabla P\Big) \cdot \partial^{s-k}\partial_2 u\, dx,\\
N_{10} =&  - w_k\int_{\mathbb{T}^2} \partial^{s-k}\partial_2 (u \cdot \nabla b) \cdot\partial^{s-k}\partial_2 b\, dx, \\
N_{11} = &\; w_k\int_{\mathbb{T}^2} \partial^{s-k}\partial_2(b \cdot \nabla u) \cdot \partial^{s-k}\partial_2 b \; dx, \\
N_{12} =& - w_k\int_{\mathbb{T}^2} \partial^{s-k}\partial_2(b \nabla \cdot u) \cdot\partial^{s-k}\partial_2 b \; dx,\\
N_{13} = & - w_k\int_{\mathbb{T}^2} \partial^{s-k}\partial_2 \Big(\frac{1}{2(\rho+1)} \nabla |b|^2 \Big)\cdot \partial^{s-k}\partial_2 u \; dx.
\end{align*}
To bound the first term $N_1$, we use the basic equalities
$$
w_0 (t) = \; w_{-1}^{\frac{1}{2}} (t) w_1^{\frac{1}{2}} (t), \qquad
w_2 (t) = \; w_{1}^{\frac{1}{2}} (t) w_3^{\frac{1}{2}} (t)
$$
to obtain, for $k=1$,
\begin{equation}\label{1N1}
\int_0^t |N_1| \; d \tau \lesssim \; \mathcal{E}_0^{\frac{1}{2}} (t) \big( \mathcal{E}_1^{\frac{1}{2}}(t) + \mathcal{P}_1^{\frac{1}{2}}(t) +\mathcal{B}_1^{\frac{1}{2}}(t)  \big).
\end{equation}
For $k=3$,
\begin{equation}\label{3N1}
\int_0^t |N_1| \; d \tau \lesssim \;   \big( \mathcal{E}_1^{\frac{1}{2}} (t)+ \mathcal{P}_1^{\frac{1}{2}}(t) +\mathcal{B}_1^{\frac{1}{2}}(t)  \big)\big( \mathcal{E}_3^{\frac{1}{2}}(t) + \mathcal{P}_3^{\frac{1}{2}}(t) +\mathcal{B}_3^{\frac{1}{2}}(t)  \big).
\end{equation}
By integration by parts,
\begin{equation}\label{N2}
N_2  = 0.
\end{equation}
To bound the rest of the terms more efficiently, we regroup them into three types:
Transport terms, the terms containing $\na\cdot u$ and good terms. We start with the transport terms.

\quad \\
\noindent
{\bf Type I: Transport terms containing $u\cdot \nabla$ or $b \cdot \nabla$ }
~\\~

This part contains $N_3, N_5, N_7, N_{10}$ and $N_{11}$. We first deal with the three terms
involving $u\cdot\na$, namely  $N_3, N_5$ and $N_{10}$. By integration by parts,
for $g=u, b, \rho$,
\begin{equation}\label{45}
  \begin{split}
    |N_3| + |N_5| + |N_{10}|
     \le  \;  & w_k\sum_{g = u, b, \rho} \Big|\int_{\mathbb{T}^2} \partial^{s-k}(\partial_2 u \cdot \nabla  g) \partial^{s-k}\partial_2 g  \; dx \Big| \\
   & + w_k \sum_{g = u, b, \rho} \Big|\int_{\mathbb{T}^2} \partial^{s-k}( u \cdot \nabla \partial_2 g) \partial^{s-k}\partial_2 g  \; dx \Big|.
  \end{split}
\end{equation}
In the case when  $k = 1$, the first part on the right hand side of \eqref{45} can be bounded as follows,
\begin{equation}\nonumber
  \begin{split}
     &w_1\sum_{g = u, b, \rho}\Big|\sum_{m = 1}^{s-1}C_{s-1}^m\int_{\mathbb{T}^2} (\partial_2 \partial^m u \cdot \nabla \partial^{s-1-m} g) \partial^{s-1}\partial_2 g  \; dx\\
     &\qquad\qquad  + \int_{\mathbb{T}^2} (\partial_2  u \cdot \nabla \partial^{s-1} g) \partial^{s-1}\partial_2 g  \; dx \Big| \\
     \lesssim & \; w_1\sum_{g = u, b, \rho} \Big( \|\partial_2 u\|_{H^{s}} \|g\|_{H^{s-1}} \|\partial_2 g\|_{H^{s-2}} + \|\partial_2 u\|_{H^2} \|g\|_{H^s} \|\partial_2 g\|_{H^{s-1}} \Big).
  \end{split}
\end{equation}
To obtain the bounds above, we have used the following estimates
\begin{equation}\nonumber
\begin{split}
   m=s-1, \quad  \int_{\mathbb{T}^2}(\partial_2\partial^{s-1} u\cdot\nabla g )\partial^{s-1}\partial_2g\; dx
   = & - \int_{\mathbb{T}^2}(\partial_2\partial^{s} u\cdot\nabla g )\partial^{s-2}\partial_2g\; dx  \\
   & - \int_{\mathbb{T}^2}(\partial_2\partial^{s-1} u\cdot\nabla \partial g ) \partial^{s-2}\partial_2g\; dx \\
   \lesssim & \; \|\partial_2 u\|_{H^{s}} \|g\|_{H^{s-1}} \|\partial_2 g\|_{H^{s-2}}, \\
m=1,\quad  \int_{\mathbb{T}^2}(\partial_2 \partial u\cdot\nabla \partial^{s-2} g )\partial^{s-1}\partial_2g\; dx \lesssim & \|\partial_2 u\|_{H^3} \|g\|_{H^s} \|\partial_2 g\|_{H^{s-1}}.
\end{split}
\end{equation}
For $k = 3$, the first part on the right hand side of \eqref{45} can be bounded by
\begin{equation}\nonumber
  \begin{split}
     w_3\sum_{g = u, b, \rho}\Big|\int_{\mathbb{T}^2} \partial^{s-2}(\partial_2 u \cdot \nabla  g) \partial^{s-4}\partial_2 g  \; dx\Big|
     \lesssim \; w_3 \|\partial_2 u\|_{H^{s-2}} \|g\|_{H^{s-1}} \|\partial_2 g\|_{H^{s-4}}.
  \end{split}
\end{equation}
For the second part on the right-hand side of \eqref{45}, we  apply Lemma \ref{prop1}
to obtain, for both $k=1$ and $k=3$,
\begin{equation}\nonumber
  \begin{split}
    & w_k  \sum_{g = u, b, \rho} \Big(\|\nabla u_1\|_{H^{[\frac{s-k}{2}] + 1}} \|\partial_1 \partial_2g\|_{H^{s-k-1}} +\|\nabla u_2\|_{H^{[\frac{s-k}{2}] + 1}} \|\partial_2^2 g\|_{H^{s-k-1}} \\
  &+ \|\nabla u_1\|_{H^{s-k-1}} \| \partial_1 \partial_2 g\|_{H^{[\frac{s-k}{2}] + 2}}+ \|\nabla u_2\|_{H^{s-k-1}} \| \partial_2^2 g\|_{H^{[\frac{s-k}{2}] + 2}}\Big) \|\partial_2 g \|_{\dot H^{s-k}} \\
& + w_k   \sum_{g = u, b, \rho} \|\nabla \cdot u\|_{H^2}\|\partial_2 g\|_{\dot H^{s-k}}^2 \\
   \lesssim & \; w_k \|\nabla u\|_{H^{s-k-1}} \sum_{g = u, \rho, b} \|\partial_2 g\|_{H^{s-k}}^2 .
  \end{split}
\end{equation}
Therefore, for $k = 1$,
\begin{equation}\label{1N3510}
\begin{split}
 \int_0^t |N_3| + |N_5| + |N_{10}| \; d \tau
\lesssim &  \sum_{g = u, b, \rho}\sup_{0\leq \tau \leq t} w_1 \|\partial_2 g\|_{H^{s-1}}^2 \int_0^t \|\nabla u\|_{H^{s-2}} \; d\tau \\
 \lesssim  & \; \mathcal{E}_1(t)\big(\mathcal{E}_2^\frac{1}{2} (t)+ \mathcal{F}_2^\frac{1}{2}(t) + \mathfrak{U}^\frac{1}{2}(t)\big).
\end{split}
\end{equation}
For $k = 3$,
\begin{equation}\label{3N3510}
\begin{split}
\int_0^t |N_3| + |N_5| + |N_{10}| \; d \tau  \lesssim  & \sum_{g = u, b, \rho} \sup_{0\leq \tau \leq t} w_3 \|\partial_2 g\|_{H^{s-3}}^2 \int_0^t \|\nabla u\|_{H^{s-4}} \; d\tau \\
\lesssim & \; \mathcal{E}_3(t) \big(\mathcal{E}_2^\frac{1}{2}(t) + \mathcal{F}_2^\frac{1}{2}(t) + \mathfrak{U}^\frac{1}{2}(t)\big).
\end{split}
\end{equation}

Next we bound the terms containing $b\cdot\nabla$ operator, namely $N_7$ and $N_{11}$. By integration by parts,
\begin{equation}\nonumber
    N_7 =\; w_k \int_{\mathbb{T}^2} \partial^{s-k - 1} \partial_2\Big(\frac{1}{\rho+1}b \cdot \nabla b \Big)\cdot \partial^{s-k+1}\partial_2 u\; dx.
\end{equation}
According to the product estimates and the fact $\|\frac{1}{\rho+1}\|_{L^\infty} \lesssim 1$,
\begin{equation}\nonumber
  \begin{split}
     |N_7|
    \lesssim & \; w_k \Big(\|\partial_2 \rho\|_{H^{s-k-1}} \|b\|_{H^{s-k}}^2 + \|\partial_2 b\|_{H^{s-k-1}} \|b\|_{H^{s-k}} \\
    & + \|b_1\|_{H^{s-k-1}} \|\partial_1 \partial_2 b\|_{H^{s-k-1}} + \|b_2\|_{H^{s-k-1}}\|\partial_2^2 b\|_{H^{s-k-1}} \Big) \|\partial_2 \nabla u\|_{H^{s-k}}.
  \end{split}
\end{equation}
Applying Poinc\'{a}re's inequality again, $|N_7|$ can be controlled by
$$ w_k \Big(\|\partial_2^2 \rho\|_{H^{s-k-1}} \|b\|_{H^{s-k}}^2 + \|\partial_2^2 b\|_{H^{s-k-1}} \|b\|_{H^{s-k}}
    + \|\partial_2^2 b\|_{H^{s-k-1}} \|\partial_2 b\|_{H^{s-k}}\Big) \|\partial_2 \nabla u\|_{H^{s-k}}.
$$
Integrating $|N_7|$ in time on $[0,t]$, for both $k = 1, 3$  we can derive the unified estimate like follows.
\begin{equation} \label{N7}
  \begin{split}
    \int_0^t |N_7| \; d\tau \lesssim & \sup_{0\leq \tau \leq t} \|b\|_{H^{s-k}}^2 \int_0^t w_k^\frac{1}{2}\|\partial_2^2\rho\|_{H^{s-k-1}} w_k^\frac{1}{2} \|\partial_2 \nabla u\|_{H^{s-k}} \; d\tau \\
    & + \sup_{0\leq \tau \leq t} (\|b\|_{H^{s-k}} + \|\partial_2 b\|_{H^{s-k}}) \int_0^t w_k^\frac{1}{2}\|\partial_2^2b\|_{H^{s-k-1}} w_k^\frac{1}{2} \|\partial_2 \nabla u\|_{H^{s-k}} \; d\tau \\
    \lesssim & \; \mathfrak{E}(t) \mathcal{P}_k^\frac{1}{2}(t) \mathcal{E}_k^\frac{1}{2}(t) + \big(\mathfrak{E}^\frac{1}{2}(t) + \mathcal{E}_1^\frac{1}{2}(t)\big)\mathcal{B}_k^\frac{1}{2}(t)\mathcal{E}_k^\frac{1}{2}(t).
  \end{split}
\end{equation}
$N_{11}$ can be estimated similarly. We first write it as
$$
 N_{11}  = w_k \int_{\mathbb{T}^2} \partial^{s-k+1} (b \cdot \nabla  u) \cdot \partial^{s-k-1}\partial_2^2 b\; dx.
$$
After a similar process, we obtain, for $k = 1$,
\begin{equation}\label{1N11}
  \begin{split}
    \int_0^t |N_{11}| \; d\tau \lesssim  & \; \sup_{0\leq \tau \leq t} w_2^{\frac{1}{2}} \| b_1\|_{L^{\infty}} \int_0^t w_0^{\frac{1}{2}} \| \partial_1 u \|_{H^s} w_1^{\frac{1}{2}} \| \partial^2 b\|_{H^{s-2}} \; d\tau \\
 & + \sup_{0\leq \tau \leq t} w_1^{\frac{1}{2}} \| b_1\|_{H^s} \int_0^t  \| \partial_1 u \|_{L^{\infty}} w_1^{\frac{1}{2}} \| \partial^2 b\|_{H^{s-2}} \; d\tau  \\
 & + \sup_{0\leq \tau \leq t} \| b_2\|_{L^{\infty}} \int_0^t w_1^{\frac{1}{2}} \| \partial_2 u \|_{H^s} w_1^{\frac{1}{2}} \| \partial^2 b\|_{H^{s-2}} \; d\tau \\
 & + \sup_{0\leq \tau \leq t} w_0^{\frac{1}{2}} \| b_2\|_{H^s} \int_0^t  w_2^{\frac{1}{2}} \| \partial_2 u \|_{L^{\infty}} w_1^{\frac{1}{2}} \| \partial^2 b\|_{H^{s-2}} \; d\tau \\
\lesssim  & \; \mathcal{E}_2^\frac{1}{2}(t)\mathcal{E}_0^\frac{1}{2}(t)\mathcal{B}_1^\frac{1}{2}(t) + \mathcal{E}_1^\frac{1}{2}(t)\mathfrak{E}^\frac{1}{2}(t) \mathcal{B}_1^\frac{1}{2}(t) \\
&+\mathfrak{E}^\frac{1}{2}(t)\mathcal{E}_1^\frac{1}{2}(t) \mathcal{B}_1^\frac{1}{2}(t) + \mathcal{E}_0^\frac{1}{2}(t)\mathcal{E}_2^\frac{1}{2}(t)\mathcal{B}_1^\frac{1}{2}(t).
  \end{split}
\end{equation}
Notice that we have used $ \|b_1\|_{H^s} \lesssim \| \partial b\|_{H^{s-1}}$ again.
When $k = 3$, we have
\begin{equation}\label{3N11}
  \begin{split}
    \int_0^t |N_{11}| \; d\tau  \lesssim & \; \sup_{0\leq \tau \leq t} w_3^{\frac{1}{2}} \| b \|_{L^{\infty}} \int_0^t \| \nabla u \|_{H^{s-2}} w_3^{\frac{1}{2}} \| \partial_2^2 b\|_{H^{s-4}}  \; d\tau \\
    & + \sup_{0\leq \tau \leq t} w_3^{\frac{1}{2}} \| b \|_{H^{s-2}} \int_0^t \| \nabla u \|_{L^{\infty}} w_3^{\frac{1}{2}} \| \partial_2^2 b\|_{H^{s-4}} \; d\tau \\
    \lesssim  & \; \mathcal{E}_3^\frac{1}{2}(t) \mathfrak{E}^\frac{1}{2}(t) \mathcal{B}_3^\frac{1}{2}(t).
  \end{split}
\end{equation}

\quad \\
\noindent
{\bf Type II: Terms containing  the divergence $\nabla \cdot u$ }
~\\~

This part contains three terms $N_4, N_9$ and $N_{12}$. Since $N_4$ and $N_{12}$ have
similar structures, they are estimated at the same time.
\begin{equation}\nonumber
  \begin{split}
    |N_4| + |N_{12}| \lesssim & \; w_k \big\|(\rho| b) \nabla \cdot u \big\|_{H^{s-k+1}} \big\|\partial_2^2 (\rho| b)\big\|_{H^{s-k-1}} \\
    \lesssim & \; w_k \big(\|(\rho| b)\|_{H^2} \|\nabla \cdot u\|_{H^{s-k+1}} + \|(\rho, b)\|_{H^{s-k+1}} \|\nabla \cdot u\|_{H^2} \big)\|\partial_2^2 (\rho| b)\|_{H^{s-k-1}}.
  \end{split}
\end{equation}
Therefore, for $k = 1$,
\begin{equation}\label{1N412}
\begin{split}
\int_0^t |N_4| + |N_{12}| \; d\tau \lesssim & \sup_{0\leq \tau \leq t} \|(\rho| b)\|_{H^2} \int_0^t w_1^\frac{1}{2}\|\nabla \cdot u\|_{H^{s}} w_1^\frac{1}{2}\|\partial_2^2 (\rho| b)\|_{H^{s-2}}\,d\tau\\
 & + \sup_{0\leq \tau \leq t} w_0^\frac{1}{2}\|(\rho| b)\|_{H^s} \int_0^t (1+\tau)^\frac{1}{2}\|\nabla \cdot u\|_{H^2}w_1^\frac{1}{2}\|\partial_2^2 (\rho| b)\|_{H^{s-2}}\,d\tau\\
 \lesssim & \; \mathfrak{E}^\frac{1}{2} (t)\big(\mathcal{E}_1^\frac{1}{2}(t) + \mathcal{F}_1^\frac{1}{2}(t)\big)\big(\mathcal{P}_1^\frac{1}{2}(t) + \mathcal{B}_1^\frac{1}{2}(t)\big) \\
 & + \mathcal{E}_0^\frac{1}{2}(t)\big(\mathcal{E}_2^\frac{1}{2}(t) + \mathcal{F}_2^\frac{1}{2}(t)\big)\big(\mathcal{P}_1^\frac{1}{2}(t) + \mathcal{B}_1^\frac{1}{2}(t)\big).
\end{split}
\end{equation}
For $k = 3$,
\begin{equation}\label{3N412}
\begin{split}
\int_0^t |N_4| + |N_{12}| \; d\tau \lesssim  & \sup_{0\leq \tau \leq t} \|(\rho| b)\|_{H^{s-2}} \int_0^t w_3^\frac{1}{2}\|\nabla \cdot u\|_{H^{s-2}} w_3^\frac{1}{2}\|\partial_2^2 (\rho| b)\|_{H^{s-4}}\,d\tau \\
 \lesssim & \; \mathfrak{E}^\frac{1}{2}(t) \big(\mathcal{E}_3^\frac{1}{2}(t) + \mathcal{F}_3^\frac{1}{2}(t)\big) \big(\mathcal{P}_3^\frac{1}{2}(t) + \mathcal{B}_3^\frac{1}{2}(t)\big).
\end{split}
\end{equation}
Next we deal with $N_9$. It can be written as
\begin{equation}\nonumber
\begin{split}
N_9 = &- w_k \int_{\mathbb{T}^2} \partial^{s-k} \partial_2 \Big( \Big(\frac{P'(\tilde \rho)}{\rho+1} - 1\Big) \nabla \rho \Big) \cdot \partial^{s-k}\partial_2 u \; dx \\
= & \;  w_k\int_{\mathbb{T}^2} \partial^{s-k-1} \partial_2^2 q_1(\rho)  \partial^{s-k+1}  \nabla \cdot u \; dx,
\end{split}
\end{equation}
where $q_1(\rho)$ is defined as before, namely
\begin{equation}\nonumber
q_1(\rho) \triangleq \int_0^\rho \Big(\frac{P'(r + 1)}{r + 1} - 1 \Big) \; dr.
\end{equation}
As explained before,  $q(\rho) \sim \rho^2$. Therefore, for both $k=1$ and $k=3$,
\begin{equation}\label{N9}
  \begin{split}
    \int_0^t |N_9| \; d\tau \lesssim & \int_0^t w_k \|\partial_2^2 q_1(\rho)\|_{H^{s-k-1}} \|\nabla \cdot u\|_{H^{s-k+1}} \; d\tau \\
    \lesssim & \sup_{0\leq \tau \leq t} \|\rho\|_{H^{s-k-1}} \int_0^t w_k^\frac{1}{2} \|\partial_2^2 \rho\|_{H^{s-k-1}} w_k^\frac{1}{2}\|\nabla \cdot u\|_{H^{s-k+1}} \; d\tau \\
    \lesssim & \; \mathfrak{E}^\frac{1}{2}(t)\mathcal{P}_k^\frac{1}{2} (t)\big(\mathcal{E}_k^\frac{1}{2}(t) + \mathcal{F}_k^\frac{1}{2}(t)\big).
  \end{split}
\end{equation}
This finishes the proof of Type II terms.

\quad \\
\noindent
{\bf Type III: Basic good terms}
~\\~

The remaining terms $N_6, N_8, N_{13}$ are called the basic good terms, which are bounded in
this part.  To bound the first term $N_6$, we set $f(\rho) = \frac{\rho}{\rho+1}$ and use \eqref{assum} to obtain
\begin{equation}\nonumber
\begin{split}
  |N_6| =  &\;  w_k \Big|\int_{\mathbb{T}^2} \partial^{s-k -1} \partial_2\Big(\frac{\rho}{\rho + 1} \Delta u \Big)\cdot \partial^{s-k +1} \partial_2 u \; dx \Big| \\
  \lesssim & \; w_k \big(\|\partial_2 f(\rho)\|_{L^\infty} \|\Delta u\|_{H^{s-k -1}} + \|\partial_2 f(\rho)\|_{H^{s-k-1}} \|\Delta u\|_{L^\infty} \big) \|  \partial_2 u\|_{H^{s-k+1}} \\
  & + w_k \big(\|f(\rho)\|_{L^\infty} \|\partial_2 \Delta u\|_{H^{s-k -1}} + \|f(\rho)\|_{H^{s-k-1}} \|\partial_2 \Delta u\|_{L^\infty} \big) \|  \partial_2  u\|_{H^{s-k+1}} \\
  \lesssim & \; w_k^\frac{1}{2} \|\partial_2 \rho\|_{H^{s-k-1}} \|\Delta u\|_{H^{s-k-1}} w_k^\frac{1}{2} \|\partial_2 u\|_{H^{s-k+1}}+ \|\rho\|_{H^{s-k-1}}w_k  \|\partial_2  u\|_{H^{s-k+1}}^2.
\end{split}
\end{equation}
This implies,  for both $k=1$ and $k=3$,
\begin{equation}\label{N6}
  \begin{split}
    \int_0^t |N_6|  \; d\tau \lesssim \; \mathfrak{E}^\frac{1}{2}(t) \mathcal{E}_k(t).
  \end{split}
\end{equation}
By integration by parts,  $N_8$ can be written as
\begin{align*}
|N_8|
= &\; w_k \Big|\int_{\mathbb{T}^2} \partial^{s-k-2} \partial_2^2\Big(\frac{\rho}{\rho + 1}  \nabla^{\bot}\cdot b \Big) \partial^{s-k+2}  u_1 \; dx \Big| \\
\lesssim & \; w_k \big(\|\partial_2^2 \rho \|_{H^{s-k-2}} \|b\|_{H^{s-k}} + \|\rho\|_{H^{s-k-1}} \|\partial_2^2 \nabla^\bot \cdot b\|_{H^{s-k-2}} \big)\\
&\times \big(\|\partial_1 u_1\|_{H^{s-k+1}} + \|\partial_2 u_1\|_{H^{s-k+1}}\big).
\end{align*}
Therefore, for both $k=1, 3$,
\begin{equation}\label{N8}
  \begin{split}
    \int_0^t |N_8|  \; d\tau \lesssim \; \mathfrak{E}^\frac{1}{2}(t) \big(\mathcal{P}_k^\frac{1}{2}(t) +  \mathcal{B}_k^\frac{1}{2}(t)\big)\big(   \mathcal{E}_k^\frac{1}{2}(t) + \mathcal{F}_k^\frac{1}{2}(t) \big).
  \end{split}
\end{equation}
For the last term $N_{13}$, we directly write
\begin{equation}\nonumber
\begin{split}
|N_{13}| = &\;  w_k \Big|\int_{\mathbb{T}^2} \partial^{s-k -1} \partial_2 \Big(\frac{1}{2(\rho + 1)}  \nabla |b|^2 \Big)\cdot \partial^{s-k +1} \partial_2 u \; dx \Big| \\
\leq & \; w_k \Big|\int_{\mathbb{T}^2} \partial^{s-k -1} \partial_2 \Big(\frac{1}{2(\rho + 1)}  \partial_2 |b|^2 \Big) \partial^{s-k +1} \partial_2 u_2 \; dx \Big| \\
& + \; w_k \Big|\int_{\mathbb{T}^2} \partial^{s-k -2} \partial_2^2 \Big(\frac{1}{2(\rho + 1)}  \partial_1 |b|^2 \Big) \partial^{s-k +2} u_1 \; dx \Big| \\
\lesssim & w_k   (1+\|\rho\|_{H^{s-k}}) \|\partial_2^2 b\|_{H^{s-k-1}} \|b\|_{H^{s-k}}  \|\partial_2 u\|_{H^{s-k+1}} \\
& + w_k\Big(  \|\partial_2 \rho\|_{H^{s-k-1}} \|b\|_{H^{s-k}}^2 + (1 + \|\rho\|_{H^{s-k-2}}) \|\partial_2^2 b\|_{H^{s-k-1}}\|b\|_{H^{s-k}} \Big) \\
& \times\big(\|\partial_1 u_1\|_{H^{s-k+1}} + \|\partial_2 u_1\|_{H^{s-k+1}} \big)
\end{split}
\end{equation}
Hence, for $k=1$ and $k=3$,
\begin{equation}\label{N13}
  \begin{split}
    \int_0^t |N_{13}|  \; d\tau  \lesssim \; \mathcal{B}_k^\frac{1}{2} (t) \mathfrak{E}^\frac{1}{2}(t) \mathcal{E}_k^\frac{1}{2}(t) + \big( \mathcal{P}_k^\frac{1}{2}(t)\mathfrak{E}(t)+\mathcal{B}_k^\frac{1}{2} (t)\mathfrak{E}^\frac{1}{2}(t) \big) \big( \mathcal{E}_k^\frac{1}{2}(t)+\mathcal{F}_k^\frac{1}{2}(t) \big),
  \end{split}
\end{equation}
where we have used \eqref{assum}.
This completes the proof of Type III terms.

\vskip .1in
Finally we combine the bounds above to obtain the desired inequalities. For $k=1$,
we integrate \eqref{N} in time and use the bounds in \eqref{1N1}, \eqref{N2}, \eqref{1N3510}, \eqref{N7}, \eqref{1N11}, \eqref{1N412}, \eqref{N9}, \eqref{N6}, \eqref{N8} and \eqref{N13} to obtain
$$ \mathcal{E}_1(t) \lesssim \; \mathcal{E}_0^{\frac{1}{2}} (t) \big( \mathcal{E}_1^{\frac{1}{2}}(t) + \mathcal{P}_1^{\frac{1}{2}}(t) +\mathcal{B}_1^{\frac{1}{2}}(t)  \big)+ \mathcal{E}^{\frac{3}{2}}_{total}(t). $$
All intermediate energies have been controlled by $\mathcal{E}_{total}$ through interpolation.

\vskip .1in
Similarly we obtain the bound for $\mathcal{E}_3(t)$ after we invoke the estimates \eqref{3N1}, \eqref{N2}, \eqref{3N3510}, \eqref{N7}, \eqref{3N11}, \eqref{3N412}, \eqref{N9}, \eqref{N6}, \eqref{N8} and \eqref{N13},
$$\mathcal{E}_3(t) \lesssim \; \big( \mathcal{E}_1^{\frac{1}{2}}(t) + \mathcal{P}_1^{\frac{1}{2}} (t)+\mathcal{B}_1^{\frac{1}{2}} (t) \big)\big( \mathcal{E}_3^{\frac{1}{2}}(t) + \mathcal{P}_3^{\frac{1}{2}}(t) +\mathcal{B}_3^{\frac{1}{2}} (t) \big) + \mathcal{E}^{\frac{3}{2}}_{total}(t). $$
This completes the proof of Proposition \ref{strongDis}.
\end{proof}

\vskip .1in
Next we turn to the proof of Proposition \ref{strongDis2}.

\begin{proof}[Proof of Proposition \ref{strongDis2}] The proof makes use of the wave structure
	 in (\ref{ww}).
The equation of $u_2$ in \eqref{mhd1} can be rewritten as
\begin{equation}\nonumber
\begin{split}
\partial_2 \rho = & \; \Big( 1- \frac{1}{\rho + 1} P'(\rho + 1)\Big) \partial_2 \rho - \partial_t u_2 + \frac{1}{\rho+1} \Delta u_2 -
u \cdot \nabla u_2  \\
&\quad + \frac{1}{\rho + 1} b \cdot \nabla b_2 - \frac{1}{2(\rho + 1)} \partial_2 |b|^2.
\end{split}
\end{equation}
Applying $\partial_2$ derivative to above equality and taking inner product with $\partial_2^2 \rho$ in $\dot H^{s-1-k} \; (k=1,3)$ yields

\begin{equation}\nonumber
\begin{split}
\big\|\partial_2^2 \rho \big\|_{\dot H^{s-1-k}}^2 = \sum_{i = 1}^{6} J_i,
\end{split}
\end{equation}
where,
\begin{align*}
    J_1 =& -\int_{\mathbb{T}^2}  \p^{s-1-k}\partial_t \partial_2 u_2\, \p^{s-1-k}\partial_2^2 \rho\, dx,\\
    J_2 =& \int_{\mathbb{T}^2}    \p^{s-1-k}\partial_2 \Big( \big( 1- \frac{1}{\rho + 1} P'(\rho + 1) \big) \partial_2 \rho \Big)\,  \p^{s-1-k}\partial_2^2 \rho\,dx,\\
    J_3 =& \int_{\mathbb{T}^2}    \p^{s-1-k}\partial_2 \Big(\frac{1}{\rho+1} \Delta u_2 \Big)\,  \p^{s-1-k}\partial_2^2 \rho\,dx,\\
    J_4 =& -\int_{\mathbb{T}^2}   \p^{s-1-k}\partial_2(u \cdot \nabla u_2)\, \p^{s-1-k}\partial_2^2 \rho\,dx,\\
    J_5 =& \int_{\mathbb{T}^2}   \p^{s-1-k}\partial_2(\frac{1}{\rho + 1} b \cdot \nabla b_2)\, \p^{s-1-k}\partial_2^2 \rho\,dx,\\
    J_6 =& -\int_{\mathbb{T}^2}  \p^{s-1-k} \partial_2\Big(\frac{1}{2(\rho + 1)} \partial_2 |b|^2 \Big)\, \p^{s-1-k}\partial_2^2 \rho\,dx.
 \end{align*}
The bounds for $J_2 - J_6$ can be derived directly,

\begin{equation}\nonumber
  \begin{split}
    \sum_{i = 2}^6 |J_i|
    \lesssim & \; \big(\|\rho\|_{H^{s-1-k}} \|\partial_2^2 \rho\|_{H^{s-1-k}} + \|\partial_2 \rho\|_{H^{s-1-k}}^2 \big) \|\partial_2 ^2 \rho\|_{H^{s-1-k}}\\
    & + \|\partial_2 \rho\|_{H^{s-1-k}} \big(\|\Delta u\|_{H^{s-1-k}} + \|b\|_{H^{s-1-k}} \|\nabla b\|_{H^{s-1-k}}\big)\|\partial_2 ^2 \rho\|_{H^{s-1-k}}\\
    & + \big(1+\|\rho\|_{H^{s-1-k}}\big)\big(\|\Delta \partial_2 u\|_{H^{s-1-k}} + \|\partial_2 b\|_{H^{s-1-k}} \|\nabla b \|_{H^{s-1-k}} \\
    & + \|b_1\|_{H^{s-1-k}}\|\partial_1 \partial_2 b\|_{H^{s-1-k}}  + \|b\|_{H^{s-1-k}} \|\partial_2^2 b\|_{H^{s-1-k}}\big) \|\partial_2^2 \rho\|_{H^{s-1-k}} \\
    & + \|  u \|_{H^{s-k}}\| \partial_2 \nabla u \|_{H^{s-1-k}}\|\partial_2^2 \rho\|_{H^{s-1-k}}.
  \end{split}
\end{equation}
Applying Poinc\'{a}re's inequality and then integrating in time yields, for $k=1$ and $k=3$,
\begin{equation}\label{J23456}
\begin{split}
   \sum_{i=2}^6 \int_0^t w_k |J_i| \; d\tau \lesssim \; & \mathfrak{E}^{\frac{1}{2}}(t)\mathcal{P}_k(t)+ \mathcal{E}_k^{\frac{1}{2}}(t)\mathcal{P}_k(t)+ \mathfrak{E}(t)\mathcal{P}_k(t) + \mathcal{E}_k^{\frac{1}{2}}(t)\mathcal{P}_k^{\frac{1}{2}}(t)\\
   & + \mathcal{B}_k^{\frac{1}{2}}(t)\mathfrak{E}^{\frac{1}{2}}(t) \mathcal{P}_k^{\frac{1}{2}}(t)+ \mathcal{B}_k^{\frac{1}{2}}(t)\mathcal{E}_k^{\frac{1}{2}}(t) \mathcal{P}_k^{\frac{1}{2}}(t) \\
   & + \mathfrak{E}^{\frac{1}{2}}(t) \mathcal{B}_k^{\frac{1}{2}}(t)\mathcal{P}_k^{\frac{1}{2}}(t)+\mathfrak{E}^{\frac{1}{2}}(t) \mathcal{E}_k^{\frac{1}{2}}(t)\mathcal{P}_k^{\frac{1}{2}}(t).
\end{split}
\end{equation}
Now we come back to deal with the first term $J_1$. It can be written into the following form.
\begin{equation}\nonumber
  \begin{split}
    J_1 =& \; -\partial_t \int_{\mathbb{T}^2}   \p^{s-1-k}\partial_2 u_2\, \p^{s-1-k}\partial_2^2 \rho\,dx + \int_{\mathbb{T}^2}   \p^{s-1-k}\partial_2 u_2\, \p^{s-1-k}\partial_2^2 \p_t\rho\,dx\\
        =& \; -\partial_t\int_{\mathbb{T}^2}   \p^{s-1-k}\partial_2 u_2\, \p^{s-1-k}\partial_2^2 \rho\,dx\\
        &- \int_{\mathbb{T}^2}   \p^{s-1-k}\partial_2 u_2\,\p^{s-1-k} \partial_2^2\big(u\cdot \nabla \rho + (\rho  + 1) \nabla \cdot u\big)\,dx.
  \end{split}
\end{equation}
Therefore,
\begin{equation}\label{4.10}
  \begin{split}
    & \int_0^t w_k(\tau) J_1 \; d\tau \\
    =& \;- w_k\int_{\mathbb{T}^2}   \p^{s-1-k}\partial_2 u_2(t)\, \p^{s-1-k}\partial_2^2 \rho(t)\,dx +
    \int_{\mathbb{T}^2}   \p^{s-1-k}\partial_2 u_2(0)\, \p^{s-1-k}\partial_2^2 \rho(0)\,dx\\
     &+ \int_0^t (k-\sigma) w_{k-1} \int_{\mathbb{T}^2}   \p^{s-1-k}\partial_2 u_2\, \p^{s-1-k}\partial_2^2 \rho\,dx\;d\tau\\
    & - \int_0^t w_k \int_{\mathbb{T}^2}   \p^{s-1-k}\partial_2 u_2\,\p^{s-1-k} \partial_2^2\big(u\cdot \nabla \rho + (\rho  + 1) \nabla \cdot u\big)\,dx \;d\tau.
  \end{split}
\end{equation}
The first three parts on the right-hand side of \eqref{4.10} is bounded by
\begin{equation}\label{J11}
  \begin{split}
   &\sup_{0\leq \tau \leq t} w_k\|\partial_2 u_2\|_{H^{s-1-k}} \|\partial_2^2 \rho\|_{ H^{s-1-k}}
   + \|\partial_2 u_2(0)\|_{H^{s-1-k}} \|\partial_2^2 \rho(0)\|_{ H^{s-1-k}}\\
    &+   \int_0^t w_{k-1}^\frac{1}{2} \|\partial_2 u_2\|_{H^{s-1-k}}w_{k-1}^\frac{1}{2}\|\partial_2^2 \rho\|_{ H^{s-1-k}} \;d\tau \\
     & \lesssim  \; \mathcal{E}_k(t)  +   \mathcal{E}_k(0)       +\mathcal{E}_k^\frac{1}{2}(t) \mathcal{P}_k^\frac{1}{2}(t).
  \end{split}
\end{equation}
The last part in \eqref{4.10} admits the bound
\begin{equation}\label{J12}
  \begin{split}
    & \int_0^t w_k \int_{\mathbb{T}^2}   \p^{s-1-k}\partial_2 \nabla u_2\, \p^{s-1-k}\partial_2^2(\rho u)\,dx\,d\tau\\
    &\quad - \int_0^t w_k \int_{\mathbb{T}^2}   \p^{s-1-k}\partial_2  u_2\, \p^{s-1-k}\partial_2^2\nabla \cdot u\,dx\;d\tau \\
    \lesssim & \; \mathfrak{E}^\frac{1}{2} (t)\mathcal{E}_k(t) + \mathfrak{E}^\frac{1}{2}(t) \mathcal{P}_k^\frac{1}{2} (t)\mathcal{E}_k^\frac{1}{2}(t)+  \mathcal{E}_k (t).
  \end{split}
\end{equation}
Combining the estimates for $|J_1|$ to $|J_6|$, namely \eqref{J23456}, \eqref{J11} and \eqref{J12}, applying Young's inequality we then complete the proof of Proposition \ref{strongDis2}.
\end{proof}

\vskip .1in
Finally we prove Proposition \ref{strongDis3}, which shows the enhanced dissipation in the vertical
derivatives of $b$.

\begin{proof}[Proof of Proposition \ref{strongDis3}]
We rewrite the equation of $u_1$ in \eqref{mhd1} as follows
\begin{equation}\label{u1}
\begin{split}
  \nabla^{\bot} \cdot b = & \; \frac{\rho}{\rho+1}\nabla^{\bot} \cdot b + \partial_t u_1 - \frac{1}{\rho+1}\Delta u_1 +  u \cdot \nabla u_1 \\
  & + \frac{1}{\rho+1}\partial_1 P -\frac{1}{\rho+1} b \cdot \nabla b_1 + \frac{1}{2(\rho+1)} \partial_1 |b|^2.
\end{split}
\end{equation}
Applying $\p_2^2$ to (\ref{u1}) and then taking the inner product in $\dot H^{s-2-k}  (k=1,3)$ yields
\begin{equation}\nonumber
  \begin{split}
    \big\|\partial_2^2 \nabla^{\bot} \cdot b \big\|^2_{\dot H^{s-2-k}} = \sum_{i = 1}^7 K_i,
  \end{split}
\end{equation}
where
\begin{align*}
    K_1 = & \int_{\mathbb{T}^2}   \p^{s-2-k}\partial_2^2 \partial_t u_1\,  \p^{s-2-k}\partial_2^2 \nabla^{\bot} \cdot b\,dx,\\
    K_2 =&  \int_{\mathbb{T}^2}   \p^{s-2-k}\partial_2^2 (\frac{\rho}{\rho+1}\nabla^{\bot} \cdot b)\, \p^{s-2-k}\partial_2^2 \nabla^{\bot} \cdot b\,dx,\\
    K_3 = &- \int_{\mathbb{T}^2}    \p^{s-2-k}\partial_2^2 (\frac{1}{\rho+1} \Delta u_1)\, \p^{s-2-k}\partial_2^2 \nabla^{\bot} \cdot b\,dx,\\
    K_4 =&  \int_{\mathbb{T}^2}    \p^{s-2-k}\partial_2^2 (u \cdot \nabla u_1)\, \p^{s-2-k}\partial_2^2 \nabla^{\bot} \cdot b\,dx,\\
    K_5 = & \int_{\mathbb{T}^2}    \p^{s-2-k}\partial_2^2 (\frac{1}{\rho+1} \partial_1 P)\, \p^{s-2-k}\partial_2^2 \nabla^{\bot} \cdot b\,dx,\\
    K_6 =&  -\int_{\mathbb{T}^2}   \p^{s-2-k}\partial_2^2 (\frac{1}{\rho+1} b\cdot \nabla b_1)\, \p^{s-2-k}\partial_2^2 \nabla^{\bot} \cdot b\,dx,\\
    K_7 = & \int_{\mathbb{T}^2}   \p^{s-2-k}\partial_2^2 (\frac{1}{2(\rho+1)} \partial_1 |b|^2)\, \p^{s-2-k}\partial_2^2 \nabla^{\bot} \cdot b\,dx.
  \end{align*}
$K_2$ through $ K_7$ can be bounded by
\begin{equation}\nonumber
  \begin{split}
   \sum_{i = 2}^7 & |K_i|
    \lesssim  \; \big(\|\rho\|_{H^{s-2-k}} \|\partial_2^2 \nabla^{\bot} \cdot b\|_{H^{s-2-k}} + \|\partial_2 \rho\|_{H^{s-2-k}}
    \|\partial_2 \nabla^{\bot} \cdot b\|_{H^{s-2-k}}\\
     & + \|\partial_2^2 \rho\|_{H^{s-2-k}} \|\nabla^{\bot} \cdot b\|_{H^{s-2-k}} \big) \|\partial_2^2 \nabla^{\bot} \cdot b\|_{H^{s-2-k}}\\
    & + \sum_{g = \Delta u_1, \partial_1 P, b \cdot \nabla b_1, \partial_1 |b|^2} \Big( \big\|\frac{1}{\rho + 1} \big\|_{H^{s-2-k}} \|\partial_2^2 g\|_{H^{s-2-k}}\\
    &\qquad\qquad\qquad \qquad\qquad  + \big\|\partial_2 \frac{1}{\rho + 1}\big\|_{H^{s-2-k}}
    \|\partial_2 g\|_{H^{s-2-k}}\\
     &\qquad\qquad\qquad \qquad\qquad  + \big\|\partial_2^2 \frac{1}{\rho + 1}\big\|_{H^{s-2-k}} \|g\|_{H^{s-2-k}}\Big)\|\partial_2^2 \nabla^{\bot} \cdot b\|_{H^{s-2-k}} \\
     & + \| u \|_{H^{s-1-k}}\|\partial_2^2  u \|_{H^{s-1-k}}\|\partial_2^2 \nabla^{\bot} \cdot b\|_{H^{s-2-k}}.
  \end{split}
\end{equation}
Applying Poinc\'{a}re's inequality, we obtain, for $k=1$ and $k=3$,
\begin{equation}\label{K234567}
  \begin{split}
    \sum_{i = 2}^7 \int_0^t w_k  |K_i| \; d\tau
    \lesssim &   \; \mathfrak{E}^\frac{1}{2}(t) \mathcal{B}_k(t) + \mathfrak{E}^\frac{1}{2}(t)\mathcal{P}_k^\frac{1}{2}(t)  \mathcal{B}_k^\frac{1}{2}(t)    \\
    & + \mathcal{E}_k^\frac{1}{2}(t)  \mathcal{B}_k^\frac{1}{2}(t) + \mathcal{P}_k^\frac{1}{2}(t)  \mathcal{B}_k^\frac{1}{2}(t) +\mathfrak{E}^\frac{1}{2}(t) \mathcal{B}_k(t) \\
    &+   \mathcal{P}_k^\frac{1}{2}(t)   \big(  \mathfrak{E}(t)+ \mathfrak{E}^\frac{1}{2}(t)      \big)  \mathcal{B}_k^\frac{1}{2}(t) + \mathfrak{E}^\frac{1}{2}(t)  \mathcal{E}_k^\frac{1}{2}(t)  \mathcal{B}_k^\frac{1}{2}(t).
  \end{split}
\end{equation}
We now focus on the first term  $K_1$. We can rewrite it as
\begin{equation}\nonumber
  \begin{split}
    & \; \frac{d}{dt} \int_{\mathbb{T}^2}   \p^{s-2-k}\partial_2^2 u_1\, \p^{s-2-k}\partial_2^2 \nabla^{\bot} \cdot b\,dx- \int_{\mathbb{T}^2}   \p^{s-2-k}\partial_2^2 u_1\,
    \p^{s-2-k}\partial_2^2 \nabla^{\bot} \cdot \p_t b\,dx\\
    = & \; \frac{d}{dt} \int_{\mathbb{T}^2}   \p^{s-2-k}\partial_2^2 u_1\, \p^{s-2-k}\partial_2^2 \nabla^{\bot} \cdot b\,dx\\
    &- \int_{\mathbb{T}^2}   \p^{s-2-k}\partial_2^2 u_1\,\p^{s-2-k}
    \partial_2^2 \big(\Delta u_1 + \nabla^{\bot} \cdot (-u\cdot \nabla b + b \cdot \nabla u - b \nabla \cdot u) \big)\,dx.
  \end{split}
\end{equation}
It then implies that, for $k=1$ and $k=3$,
\begin{equation}\nonumber
  \begin{split}
    & \int_0^t w_k \; K_1\; d\tau  \\
    =& \; w_k \int_{\mathbb{T}^2} \p^{s-2-k}\partial_2^2 u_1(t)\, \p^{s-2-k}\partial_2^2 \nabla^{\bot} \cdot b(t)\,dx\\ &- \int_{\mathbb{T}^2} \p^{s-2-k}\partial_2^2 u_1(0)\, \p^{s-2-k}\partial_2^2 \nabla^{\bot} \cdot b(0)\,dx\\
    &-  \int_0^t  (k-\sigma) w_{k-1} \int_{\mathbb{T}^2}  \p^{s-2-k}\partial_2^2 u_1\, \p^{s-2-k}\partial_2^2 \nabla^{\bot} \cdot b\,dx\; d\tau \\
     & -\int_0^t w_k \int_{\mathbb{T}^2}  \p^{s-2-k}\partial_2^2 u_1\,\p^{s-2-k}
    \partial_2^2 \big(\Delta u_1 + \nabla^{\bot} \cdot (-u\cdot \nabla b + b \cdot \nabla u - b \nabla \cdot u) \big)\,dx \; d\tau.
  \end{split}
\end{equation}
The first three parts above can be controlled by
\begin{equation}\label{K11}
  \begin{split}
& \sup_{0\leq \tau \leq t} w_k\|\partial_2^2 u\|_{H^{s-2-k}} \|\partial_2^2 \nabla^{\bot} \cdot b\|_{H^{s-2-k}} +\|\partial_2^2 u(0)\|_{H^{s-2-k}} \|\partial_2^2 \nabla^{\bot} \cdot b(0)\|_{H^{s-2-k}}  \\
&+\int_0^t w_{k-1}^\frac{1}{2} \|\partial_2^2 u\|_{H^{s-2-k}} w_{k-1}^\frac{1}{2} \|\partial_2^2 \nabla^{\bot} \cdot b\|_{H^{s-2-k}} \; d\tau\\
    &\lesssim \mathcal{E}_k(t) + \mathcal{E}_k(0) +  \mathcal{E}_k^\frac{1}{2}(t)\mathcal{B}_k^\frac{1}{2}(t).
  \end{split}
\end{equation}
To bound the last part, we use the inequality $\|\partial_2^2 b\|_{H^{s-1-k}} \lesssim \|\partial_2^2 \nabla^{\bot} \cdot b\|_{H^{s-2-k}}$ to obtain
\begin{equation}\label{K12}
  \begin{split}
    & \int_0^t w_ k\|\partial_2^2 \nabla u\|_{H^{s-2-k}}^2   \; d\tau   +  \int_0^t   w_k\|\partial_2^2 \nabla u\|_{H^{s-2-k}}
    \big(\|u\|_{H^{s-1-k}} \|\partial_2^2 b\|_{H^{s-1-k}} \\
    &+ \|\partial_2 u\|_{H^{s-1-k}} \|\partial_2 b\|_{H^{s-1-k}}
    + \|\partial_2^2 u\|_{H^{s-1-k}} \|b\|_{H^{s-1-k}}\big) \; d\tau\\
    & \lesssim  \mathcal{E}_k (t) + \mathcal{E}_k^\frac{1}{2}(t)\big(\mathfrak{E}^\frac{1}{2}(t)\mathcal{B}_k^\frac{1}{2}(t) + \mathcal{E}_k^\frac{1}{2} (t)\mathfrak{E}^\frac{1}{2}(t)\big).
  \end{split}
\end{equation}
Finally, taking all the estimates above into consideration, namely \eqref{K234567}, \eqref{K11} and \eqref{K12}, applying Young's inequality we will complete the proof of Proposition \ref{strongDis3}.
\end{proof}

\vskip .3in
\section{Dissipative structure of divergence part}
\label{divp}

This section presents upper bounds on $\tilde{\mathcal{A}}$, $\mathcal{A}_k$ and $\mathcal{F}_k$.
As aforementioned in the introduction, the coupling of $\Omega$ (defined in (\ref{OO})) and $u_1$
generates enhanced dissipation. The estimates presented in this section reflect this extra
regularization. In particular, the bound for $\mathcal{F}_k$ verifies the regularity and
decay properties of $\p_1 u_1$. Together with the bounds on $\p_2 u_2$ in $\mathcal{E}_k$, the estimates presented here confirm the extra dissipation of $\na\cdot u$.

\vskip .1in
The main bounds are presented in the following two propositions.

\begin{proposition}\label{A}
	Let $\mathcal{E}_{total}(t)$ be the total energy defined in \eqref{total}. Then, for any $t>0$, 	
\begin{equation}\nonumber
\begin{split}
\tilde{\mathcal{A}}(t)  \lesssim &\; \mathcal{E}_0 (t) +  \mathcal{E}^{\frac{3}{2}}_{total}(t),     \\
\mathcal{A}_k(t) \lesssim  &\;  \mathcal{E}_k (t) + \mathcal{F}_k(t) +\mathcal{E}^{\frac{3}{2}}_{total}(t),  \qquad  k=1,3.
\end{split}
\end{equation}
\end{proposition}

\begin{proposition}\label{F}
	Let $\mathcal{E}_{total}(t)$ be the total energy defined in \eqref{total}. Then, for any $t>0$,
	\begin{equation}\nonumber
		\begin{split}
			\mathcal{F}_k(t) \lesssim \;    \mathcal{E}_0(t) +\tilde{\mathcal{A}}(t) +  \mathcal{E}_k^\frac{1}{2} (t) \mathcal{A}_k^\frac{1}{2}(t) + \mathcal{E}_{total}^{\frac{3}{2}}(t) , \qquad k=1,3.
		\end{split}
	\end{equation}
\end{proposition}

\vskip .1in
We now prove these bounds and start with Proposition \ref{A}.

\begin{proof}[Proof of Proposition \ref{A}]
As we explained in the proofs of the previous propositions, it suffices to deal with the
highest-order norms. According to \eqref{equ1},
\beq\label{Oeq}
    \Omega = \partial_t u_1 + u \cdot \nabla u_1 - \frac{1}{\rho + 1}\Delta u_1 + \frac{\rho}{\rho+1}\Omega .
\eeq
Let $s_0 \geq 3$. Taking the inner product of (\ref{Oeq}) with $\Omega$ in $\dot H^{s_0}$ yields
\begin{equation}\nonumber
  \|\Omega\|_{\dot H^{s_0}}^2 = \sum_{i = 1}^4 M_i,
\end{equation}
where
\begin{equation}\nonumber
\begin{split}
M_1 = & \int_{\mathbb{T}^2} \p^{s_0} \partial_t u_1\, \p^{s_0}\Omega\,dx, \quad
M_2 = \int_{\mathbb{T}^2} \p^{s_0} (u\cdot \nabla u_1)\, \p^{s_0}\Omega\,dx,\\
M_3 = & - \int_{\mathbb{T}^2} \p^{s_0}\left(\frac{1}{\rho+1}\Delta u_1\right)\, \p^{s_0}\Omega\,dx, \quad
M_4 =  \int_{\mathbb{T}^2} \p^{s_0}\left(\frac{\rho}{\rho+1}\Omega\right)\, \p^{s_0}\Omega\,dx.
\end{split}
\end{equation}
$M_2$ through $M_4$ can be bounded directly by
\begin{equation}\nonumber
\begin{split}
&|M_2| + |M_3| + |M_4|\\
 \lesssim
& \; \|u\|_{H^{s_0}}\|\nabla u_1\|_{H^{s_0}} \|\Omega\|_{H^{s_0}}+ \big(1+\|\rho\|_{H^{s_0}}\big) \|\nabla u_1\|_{H^{s_0+1}} \|\Omega\|_{H^{s_0}} + \|\rho\|_{H^{s_0}} \|\Omega\|_{H^{s_0}}^2 \\
\lesssim & \; \big(\|u\|_{H^{s_0}}+ 1+\|\rho\|_{H^{s_0}}\big)\|\nabla u_1\|_{H^{s_0+1}} \|\Omega\|_{H^{s_0}}  +  \|\rho\|_{H^{s_0}} \|\Omega\|_{H^{s_0}}^2.
\end{split}
\end{equation}
Letting $s_0 = s - 1$ and multiplying by the time weight $w_0$, we obtain
\begin{equation}\label{0M234}
  \begin{split}
    \sum_{i = 2}^{4} \int_0^t w_0|M_i| \; d\tau \lesssim & \; \sup_{0\leq \tau \leq t} \big(\|u\|_{H^{s-1}}+ 1+\|\rho\|_{H^{s-1}}\big) \int_0^t w_0^\frac{1}{2} \|\nabla u_1\|_{H^{s}} w_0^\frac{1}{2}\|\Omega\|_{H^{s-1}}\; d\tau \\
    & +  \sup_{0\leq \tau \leq t} \|\rho\|_{H^{s-1}} \int_0^t w_0 \|\Omega\|_{H^{s-1}}^2 \; d\tau\\
    \lesssim & \; \mathcal{E}_0^\frac{1}{2}(t) \tilde{\mathcal{A}}^\frac{1}{2}(t) + \mathfrak{E}^\frac{1}{2}(t)\tilde{\mathcal{A}}(t).
  \end{split}
\end{equation}
Similarly, if we use weight $w_1$, it becomes
\begin{equation}\label{1M234}
  \begin{split}
    \sum_{i = 2}^{4} \int_0^t w_1|M_i| \; d\tau \lesssim
 \;  \big(\mathcal{E}_1^\frac{1}{2}(t)+ \mathcal{F}_1^\frac{1}{2}(t)\big)\mathcal{A}_1^\frac{1}{2}(t)  + \mathfrak{E}^\frac{1}{2}(t)\mathcal{A}_1(t).
  \end{split}
\end{equation}
By taking $s_0 = s - 3$ and choosing the weight $w_3$, we have

\begin{equation}\label{3M234}
  \begin{split}
    \sum_{i = 2}^{4} \int_0^t w_3 |M_i| \; d\tau
    \lesssim  \; \big(\mathcal{E}_3^\frac{1}{2}(t) + \mathcal{F}_3^\frac{1}{2}(t)\big)\mathcal{A}_3^\frac{1}{2}(t)  + \mathfrak{E}^\frac{1}{2}(t)\mathcal{A}_3(t).
  \end{split}
\end{equation}
We now deal with the first term $M_1$. Notice the equality \eqref{eqomega} $M_1$ can be written as
\begin{equation}\nonumber
  \begin{split}
    M_1
    =  \; \frac{d}{dt} \int_{\mathbb{T}^2} \p^{s_0}u_1\, \p^{s_0}\Omega\,dx - \sum_{i = 1}^6 L_i,
  \end{split}
\end{equation}
where,
\begin{equation}\nonumber
\begin{split}
L_1 =& \int_{\mathbb{T}^2}  \p^{s_0} u_1\,  \p^{s_0}(-2\partial_1^2 u_1 -\partial_2^2 u_1 - \partial_1\partial_2 u_2)\, dx,\\
L_2 =&  \int_{\mathbb{T}^2}  \p^{s_0}u_1\, \p^{s_0}(u \cdot \nabla \Omega)\,dx, \\
L_3 = & \int_{\mathbb{T}^2} \p^{s_0}u_1\,  \p^{s_0}\Big(\partial_2 u \cdot \nabla b_1 + \partial_1 (b \cdot \nabla^{\bot} u_1) - b \cdot \nabla \partial_2 u_1 \\
&\qquad\qquad \quad- \nabla^{\bot} u_1 \cdot \nabla b_1 - b \cdot \nabla(b \cdot \nabla u_1)\Big)\,dx, \\
L_4 = & \int_{\mathbb{T}^2}  \p^{s_0} u_1\, \p^{s_0}\Big(-\partial_1 u\cdot \nabla b_2 + \frac{1}{2} \partial_1 u \cdot \nabla |b|^2 - \partial_1 u \cdot \nabla P\Big)\,dx, \\
L_5 = &  \int_{\mathbb{T}^2}\p^{s_0}  u_1\,  \p^{s_0}\Big(  - \nabla^{\bot} \cdot (b \nabla \cdot u) - \partial_1 \big(|b|^2 \nabla \cdot u \big) - \partial_1 \big((P'(\tilde \rho) \tilde \rho - 1) \nabla \cdot u \big) \\
&\qquad \qquad\qquad + \nabla \cdot u  ( b \cdot \nabla b_1)+  b \cdot \nabla (b_1 \nabla \cdot u)\Big)\,dx, \\
L_6 = & \int_{\mathbb{T}^2} \p^{s_0} u_1\, \p^{s_0}\Big(  \nabla^{\bot} \cdot (b \cdot \nabla u) + \partial_1 \big(b \cdot (b \cdot \nabla u) \big)\Big)\,dx.
\end{split}
\end{equation}
By integration by parts and basic inequalities,
$$  |L_1| \lesssim  \; \|\nabla u_1\|_{H^{s_0}} \big(\|\nabla u_1\|_{H^{s_0}} + \|\partial_2 u_2\|_{H^{s_0}}\big),$$
\begin{equation}\nonumber
  \begin{split}
    L_2 \lesssim \;  \| \nabla u_1\|_{\dot H^{s_0}} \| u\|_{H^{s_0-1}} \|\Omega\|_{H^{s_0}}.
  \end{split}
\end{equation}
For $s_0 = s - 1$ and the time weight $w_0$, we have
\begin{equation}\label{0L12}
  \begin{split}
    \int_0^t w_0 \big(|L_1| + |L_2|\big) \; d\tau  \lesssim & \int_0^t w_0 \| \nabla u\|_{H^{s-1}}^2 \; d\tau \\
    & +\sup_{0\leq \tau \leq t}  \|u\|_{H^{s-1}} \int_0^t w_0^\frac{1}{2} \|\nabla u_1\|_{H^{s-1}} w_0^\frac{1}{2}\|\Omega\|_{H^{s-1}} \; d\tau\\
    \lesssim & \;\mathcal{E}_0 (t)+ \mathfrak{E}^\frac{1}{2} (t)\mathcal{E}_0^\frac{1}{2}(t) \tilde{\mathcal{A}}^\frac{1}{2}(t).
  \end{split}
\end{equation}
Similarly, if we use weight $w_1$, we have
\begin{equation}\label{1L12}
  \begin{split}
    \int_0^t w_1 (|L_1| + |L_2|) \; d\tau
    \lesssim  \;  \mathcal{E}_1(t) + \mathcal{F}_1(t)+  \mathfrak{E}^\frac{1}{2}(t) \big(\mathcal{E}_1^\frac{1}{2}(t)+\mathcal{F}_1^\frac{1}{2}(t) \big)\mathcal{A}_1^\frac{1}{2}(t).
  \end{split}
\end{equation}
For $s_0 = s - 3$ and the time weight $w_3$, the bound is given by
\begin{equation}\label{3L12}
  \begin{split}
    \int_0^t w_3 (|L_1| + |L_2|) \; d\tau
    \lesssim  \; \mathcal{E}_3(t) + \mathcal{F}_3(t) + \mathfrak{E}^\frac{1}{2}(t) \big(\mathcal{E}_3^\frac{1}{2}(t)+\mathcal{F}_3^\frac{1}{2}(t) \big)\mathcal{A}_3^\frac{1}{2}(t).
  \end{split}
\end{equation}
 After integration by parts,  $L_3$ is bounded by
\begin{equation}\nonumber
\begin{split}
| L_3 | \lesssim & \; \|\nabla u_1\|_{H^{s_0}} \big(  \|\nabla u_1\|_{H^{s_0}} + \| \partial_2 u\|_{H^{s_0}}     \big) \| b\|_{H^{s_0}}.
\end{split}
\end{equation}
Setting $s_0 = s - 1$ with the weight $w_0$ and $w_1$ yields
\begin{equation}\label{01L3}
\begin{split}
  \int_0^t w_0 |L_3| \; d\tau
  \lesssim  \; \mathfrak{E}^\frac{1}{2}(t)\mathcal{E}_0(t),  \quad
 \int_0^t w_1 |L_3| \; d\tau
  \lesssim   \;  \mathfrak{E}^\frac{1}{2}(t)  \big( \mathcal{E}_1 (t)+ \mathcal{F}_1(t) \big).
\end{split}
\end{equation}
For $s_0 = s - 3$ and the weight $w_3$, we have
\begin{equation}\label{3L3}
\begin{split}
  \int_0^t w_3 |L_3| \; d\tau
  \lesssim  \;  \mathfrak{E}^\frac{1}{2}(t)(\mathcal{E}_3(t) + \mathcal{F}_3(t)).
\end{split}
\end{equation}
We move on to the estimate for $L_4$, which is bounded by
\begin{equation}\nonumber
  \begin{split}
    |L_4|
    \lesssim & \|\nabla u_1\|_{H^{s_0}} \Big\{\|\partial _1 u_1\|_{H^{s_0 - 1}} \big(\|b\|_{H^{s_0}} + \|b\|_{H^{s_0}}^2 + \|\rho\|_{H^{s_0}}\big)\\
    & + \|\partial_1 u_2\|_{H^{s_0-1}}\big(\|\partial_2 b\|_{H^{s_0-1}} + \|b\|_{H^{s_0-1}}\|\partial_2 b\|_{H^{s_0-1}} + \|\partial_2 \rho\|_{H^{s_0-1}}\big)\Big\}.
  \end{split}
\end{equation}
Therefor, for $s_0 = s-1$ we obtain
\begin{equation}\label{01L4}
  \begin{split}
     \int_0^t w_0 |L_4| \; d\tau   \lesssim  \; & \mathfrak{E}^\frac{1}{2}(t)\mathcal{E}_0(t), \\
    \int_0^t w_1 |L_4| \; d\tau   \lesssim  \;& \mathfrak{E}^\frac{1}{2}(t)\big(\mathcal{E}_1(t)+\mathcal{F}_1(t) \big)+ \mathfrak{E}^\frac{1}{2}(t)\big(\mathcal{E}^\frac{1}{2}_1(t)+\mathcal{F}^\frac{1}{2}_1(t)\big)\big(\mathcal{P}^\frac{1}{2}_1(t)+\mathcal{B}^\frac{1}{2}_1(t)\big).
  \end{split}
\end{equation}
For $s_0 = s-3$ and the weight $w_3$, we have
\begin{equation}\label{3L4}
  \begin{split}
    \int_0^t w_3 |L_4| \; d\tau
    \lesssim  \;\mathfrak{E}^\frac{1}{2}(t)\big(\mathcal{E}_3(t)+\mathcal{F}_3(t) \big)+ \mathfrak{E}^\frac{1}{2}(t)\big(\mathcal{E}^\frac{1}{2}_3(t)+\mathcal{F}^\frac{1}{2}_3(t)\big)\big(\mathcal{P}^\frac{1}{2}_3(t)+\mathcal{B}^\frac{1}{2}_3(t)\big).
  \end{split}
\end{equation}
$L_5$ and $L_6$ are bounded by
\begin{equation}\nonumber
  \begin{split}
   & \; |L_5| + |L_6| \\
    \lesssim & \; \|\nabla u_1\|_{H^{s_0 + 1}} \big(\|b\|_{H^{s_0 - 1}} + \|b\|_{H^{s_0 - 1}}\|b\|_{H^{s_0-2}} + \|\rho\|_{H^{s_0 - 1}}\big) \|\nabla \cdot u\|_{H^{s_0 - 1}} \\
    & + \|\nabla u_1 \|_{H^{s_0 + 1}} \big(1+ \|b\|_{H^{s_0 - 1}}\big)\big(\|b_1\|_{H^{s_0 - 1}} \|\partial_1 u\|_{H^{s_0 - 1}} + \|b_2\|_{H^{s_0 - 1}} \|\partial_2 u\|_{H^{s_0 - 1}}\big).
  \end{split}
\end{equation}
Therefore, if we take $s_0=s-1$ there will hold,
\begin{equation}\label{0L56}
  \begin{split}
    \int_0^t w_0(|L_5| + |L_6|) \; d\tau
    \lesssim  \; \mathfrak{E}^\frac{1}{2}(t) \mathcal{E}_0(t).
  \end{split}
\end{equation}
\begin{equation}\label{1L56}
  \begin{split}
    \int_0^t w_1(|L_5| + |L_6|) \; d\tau
    \lesssim  \; \mathfrak{E}^\frac{1}{2} (t)\big(\mathcal{E}_1(t) + \mathcal{F}_1 (t)\big) +\mathfrak{E}^\frac{1}{2} (t)\big(\mathcal{E}^\frac{1}{2}_1(t) + \mathcal{F}^\frac{1}{2}_1 (t)\big)\mathcal{B}^\frac{1}{2}_1(t).
  \end{split}
\end{equation}
If we take $s_0=s-3$ there will be then
\begin{equation}\label{3L56}
  \begin{split}
    \int_0^t w_3(|L_5| + |L_6|) \; d\tau
    \lesssim  \; \mathfrak{E}^\frac{1}{2}(t) \big(\mathcal{E}_3(t) + \mathcal{F}_3(t) \big)+\mathfrak{E}^\frac{1}{2} (t)\big(\mathcal{E}^\frac{1}{2}_3(t) + \mathcal{F}^\frac{1}{2}_3 (t)\big)\mathcal{B}^\frac{1}{2}_3(t).
  \end{split}
\end{equation}
Combining the estimates above allows us to prove the desired bounds for $\tilde{\mathcal{A}}(t)$,
$\mathcal{A}_1(t)$ and $\mathcal{A}_3(t)$. In fact,  \eqref{0M234}, \eqref{0L12}, \eqref{01L3}, \eqref{01L4} and \eqref{0L56} together proves the inequality $\tilde{\mathcal{A}}(t)$. \eqref{1M234}, \eqref{1L12}, \eqref{01L3}, \eqref{01L4} and \eqref{1L56} yields the inequality for $\mathcal{A}_1(t)$ while \eqref{3M234}, \eqref{3L12}, \eqref{3L3}, \eqref{3L4}, \eqref{3L56} imply the estimate for $\mathcal{A}_3(t)$. This completes the proof of Proposition \ref{A}.
\end{proof}

We move on to prove Proposition \ref{F}.

\begin{proof}[Proof of Proposition \ref{F}]
According to \eqref{equ1} and \eqref{eqomega},
\begin{equation}\label{Q}
  \frac{1}{2}\frac{d}{dt} w_k \big(2 \|\partial_1 u_1\|_{\dot H^{s-k}}^2 + \|\Omega\|_{\dot H^{s-k}}^2 \big) +  2 w_k
\|\nabla \partial_1 u_1\|_{\dot H^{s-k}}^2 = \sum_{i = 1}^{10} Q_i,
\end{equation}
where,
\begin{align*}
    Q_1 = & \; \frac{k-\sigma}{2} w_{k-1} \big(2 \|\partial_1 u_1\|_{\dot H^{s-k}}^2 + \|\Omega\|_{\dot H^{s-k}}^2\big), \\
    Q_2 =& 2 w_k \int \p^{s-k}\partial_1 \Omega\, \p^{s-k}\partial_1 u_1\,dx + 2 w_k \int \p^{s-k}\partial_1^2 u_1\, \p^{s-k}\Omega\,dx, \\
    Q_3 = & \;w_k \int \p^{s-k}(\partial_2^2 u_1 + \partial_1 \partial_2 u_2)\, \p^{s-k}\Omega\,dx,\\
    Q_4 =& - w_k \int \p^{s-k}\partial_1(u \cdot \nabla u_1)\, \p^{s-k}\partial_1 u_1\,dx,\\
    Q_5 = & - w_k \int \p^{s-k}\partial_1 \Big(\frac{\rho}{\rho+1}(\Delta u_1 + \Omega)\Big)\, \p^{s-k}\partial_1 u_1\,dx, \\
    Q_6 =& - w_k \int \p^{s-k}(u \cdot \nabla \Omega)\, \p^{s-k}\Omega\,dx, \\
    Q_7 = & - w_k \int \p^{s-k}\Big(\partial_2 u \cdot \nabla b_1 - \partial_1 (b \cdot \nabla^{\bot} u_1)  + b \cdot \nabla \partial_2 u_1\\
    & + \nabla^{\bot} u_1 \cdot \nabla b_1 + b \cdot \nabla(b \cdot \nabla u_1)\Big)\, \p^{s-k}\Omega\,dx,\\
    Q_8 = & \;w_k\int \p^{s-k}\Big(\partial_1 u\cdot \nabla b_2 - \frac{1}{2} \partial_1 u \cdot \nabla |b|^2 + \partial_1 u \cdot \nabla P\Big)\, \p^{s-k}\Omega\,dx, \\
    Q_9 = & \; w_k \int\p^{s-k} \Big(-\nabla^{\bot} \cdot (b \nabla \cdot u) + \partial_1 (|b|^2 \nabla \cdot u) + \partial_1 ((P'(\tilde \rho) \tilde \rho - 1) \nabla \cdot u) \\
    &\qquad\qquad\qquad\qquad  - \nabla \cdot u b \cdot \nabla b_1 - b \cdot \nabla (b_1 \nabla \cdot u)\Big)\, \p^{s-k}\Omega\,dx, \\
    Q_{10} = & \;w_k \int \p^{s-k}\Big(\nabla^{\bot} \cdot (b \cdot \nabla u) - \partial_1 (b \cdot (b \cdot \nabla u))\Big)\, \p^{s-k}\Omega\,dx.
  \end{align*}
For $k=1$  and $k=3$, the time integral of $|Q_1|$ is bounded by
\begin{equation}\label{Q1}
\begin{split}
  & \int_0^t |Q_1| \; d\tau  \lesssim \;  \mathcal{E}_0(t) + \tilde{\mathcal{A}}(t),\quad k=1, \\
  \int_0^t |Q_1| \; d\tau  \lesssim \; & \mathcal{F}_2(t) + \mathcal{A}_2(t)   \lesssim \;  \mathcal{F}_1^\frac{1}{2}(t)\mathcal{F}_3^\frac{1}{2}(t) + \mathcal{A}_1^\frac{1}{2}(t)\mathcal{A}_3^\frac{1}{2}(t), \quad k=3.
\end{split}
\end{equation}
Integration by parts yields $Q_2=0$ while $Q_3$ through $Q_6$ are bounded by
\begin{equation}\nonumber
  \begin{split}
    Q_3 &\lesssim  \;w_k \|\nabla \partial_2 u\|_{H^{s-k}} \|\Omega\|_{H^{s-k}}, \\
    Q_4 &\lesssim  \;w_k \|u\|_{H^{s-k}} \|\nabla  u_1\|_{H^{s-k}} \|\partial_1 u_1\|_{H^{s+1-k}}, \\
    Q_5 &\lesssim  \;w_k \|\rho\|_{H^{s-k}}\big(\|\Delta u_1\|_{H^{s-k}} + \|\Omega\|_{H^{s-k}}\big) \|\partial_1 u_1\|_{H^{s+1-k}}, \\
    Q_6 &\lesssim  \;w_k \|u\|_{H^{s-k}}\|\Omega\|_{H^{s-k}}^2.
  \end{split}
\end{equation}
Therefore, for $k=1, 3$,
\begin{equation}\label{Q23456}
  \begin{split}
    \sum_{i = 2}^6 \int_0^t|Q_i| \; d\tau
    \lesssim  \; &\mathcal{E}_k^\frac{1}{2}(t) \mathcal{A}_k^\frac{1}{2}(t) + \mathfrak{E}^\frac{1}{2}(t)\big(\mathcal{E}_k^\frac{1}{2}(t)
     + \mathcal{F}_k^\frac{1}{2}(t)
     + \mathcal{A}_k^\frac{1}{2}(t)\big) \mathcal{F}_k^\frac{1}{2}(t) \\
     &+ \mathfrak{E}^\frac{1}{2}(t)\mathcal{A}_k(t).
  \end{split}
\end{equation}
Using the similar method as in proving Lemma \ref{prop1}, $Q_7$ is bounded by
\begin{align*}
    |Q_7| \lesssim &\; w_k \|\Omega\|_{H^{s-k}}\Big\{\big(\|\partial_2 u\|_{H^{s-k}} + \|\nabla u_1\|_{H^{s-k}}\big)\|\nabla b_1\|_{H^{s-k}}+  \|b\|_{H^{s-k}} \|\nabla^2 u_1\|_{H^{s-k}} \\
    & + \|\partial_1 b\|_{H^{[\frac{s-k}{2}] + 2}} \|\nabla u_1\|_{H^{s-k}} + \|\nabla u_1\|_{H^{[\frac{s-k}{2}] + 2}} \|\partial_1 b\|_{H^{s-k}}\\
    & +  \|b\|_{H^{s-k}} \big(\|b\|_{H^{s-k}} \|\nabla^2 u_1\|_{H^{s-k}} + \|\nabla b\|_{H^{[\frac{s-k}{2}] + 2}} \|\nabla u_1\|_{H^{s-k}} \\
    & + \|\nabla u_1\|_{H^{[\frac{s-k}{2}] + 2}} \|\nabla b\|_{H^{s-k}}\big)\Big\}.
  \end{align*}
Therefore,  for $k=1, 3$,
\begin{equation}\label{Q7}
  \begin{split}
    \int_0^t |Q_7| \; d\tau &
     \lesssim \;  \big( \mathcal{E}_k^\frac{1}{2}(t)+ \mathfrak{E}^\frac{1}{2} (t)\big)   \big(\mathcal{E}_k^\frac{1}{2} (t)+ \mathcal{F}_k^\frac{1}{2}(t)\big) \mathcal{A}_k^\frac{1}{2}(t)
    +  \mathcal{E}_0^\frac{1}{2}(t)   \big(\mathcal{E}_2^\frac{1}{2} (t)+ \mathcal{F}_2^\frac{1}{2}(t)\big) \mathcal{A}_k^\frac{1}{2}(t) \\
    & +\mathfrak{E} (t)\big(\mathcal{E}_k^\frac{1}{2} (t)+ \mathcal{F}_k^\frac{1}{2}(t)\big) \mathcal{A}_k^\frac{1}{2}(t)
    +  \mathfrak{E}^\frac{1}{2} (t)\mathcal{E}_0^\frac{1}{2}(t)   \big(\mathcal{E}_2^\frac{1}{2} (t)+ \mathcal{F}_2^\frac{1}{2}(t)\big) \mathcal{A}_k^\frac{1}{2}(t).
 \end{split}
\end{equation}
$Q_8$ and $Q_9$ can be dealt with similar method. For $Q_8$ it holds that,
\begin{equation}\nonumber
\begin{split}
  |Q_8| \lesssim &\Big( \| \p_1u_1\|_{H^{s-k}}\|\p_1b_2\|_{H^{[\frac{s-k}{2}] + 2}} + \|\p_1u_1\|_{H^{[\frac{s-k}{2}] + 2}}\|\p_1b_2\|_{H^{s-k}} + \|\p_1u_2\|_{H^{s-k}}\|\p_2b_2\|_{H^{s-k}}\\
  &+ \|\p_1u\|_{H^{s-k}}\|\nabla b\|_{H^{[\frac{s-k}{2}] + 2}}\|b\|_{H^{[\frac{s-k}{2}] + 2}}+\|\p_1u\|_{H^{[\frac{s-k}{2}] + 2}} \|\nabla b\|_{H^{s-k}}\|b\|_{H^{[\frac{s-k}{2}] + 2}}   \\
  &+\|\p_1 u\|_{H^{s-k}}\|\nabla \rho\|_{H^{[\frac{s-k}{2}] + 2}}+ \|\p_1u\|_{H^{[\frac{s-k}{2}] + 2}}\|\nabla\rho\|_{H^{s-k}}\Big)\| \Omega\|_{H^{s-k}} .
\end{split}
\end{equation}
Therefor, for $k=1,3$ there is,
\begin{equation}\label{Q8}
  \begin{split}
    \int_0^t|Q_8| \; d\tau
    \lesssim  &\; \big(\mathfrak{E}^\frac{1}{2}(t) \mathcal{F}_k^\frac{1}{2}(t)+ \mathcal{F}_2^\frac{1}{2}(t)\mathcal{E}_0^\frac{1}{2}(t) \big)\mathcal{A}_k^\frac{1}{2} (t)\\
    &+\big(\mathfrak{E}^\frac{1}{2}(t) \mathcal{B}_k^\frac{1}{2}(t)+ \mathcal{E}_0^\frac{1}{2}(t)\mathcal{E}_2^\frac{1}{2}(t) \big)\mathcal{A}_k^\frac{1}{2} (t)\\
    &+\big(\mathcal{E}_0^\frac{1}{2}(t)\mathfrak{E}^\frac{1}{2}(t) \mathcal{B}_2^\frac{1}{2}(t)+\mathfrak{E}(t) \mathcal{B}_k^\frac{1}{2}(t)\big)\mathcal{A}_k^\frac{1}{2} (t)\\
   & + \Big( \big(\mathcal{E}_k^\frac{1}{2}(t)+\mathcal{F}_k^\frac{1}{2}(t)\big)\mathfrak{E}^\frac{1}{2}(t)+\big(\mathcal{E}_2^\frac{1}{2}(t)+\mathcal{F}_2^\frac{1}{2}(t)\big)\mathcal{E}_0^\frac{1}{2}(t)\Big)\mathcal{A}_k^\frac{1}{2} (t).
  \end{split}
\end{equation}
And for $Q_9$ we can write,

\begin{equation}\nonumber
\begin{split}
  &|Q_9|\lesssim  \Big( \|\nabla b\|_{H^{[\frac{s-k}{2}] + 2}}\|\nabla\cdot u\|_{H^{s-k}} + \|\nabla b\|_{H^{s-k}}\|\nabla\cdot u\|_{H^{[\frac{s-k}{2}] + 2}}
   + \|b\|_{H^{s-k}} \|\nabla \nabla\cdot u\|_{H^{s-k}} \\
  &+ \|\p_1\nabla\cdot u\|_{H^{s-k}}\|b\|_{H^{s-k}}^2
  + \|\nabla b\|_{H^{s-k}}\|b\|_{H^{[\frac{s-k}{2}] + 2}}\|\nabla\cdot u\|_{H^{[\frac{s-k}{2}] + 2}}\\
  &+ \|\nabla b\|_{H^{[\frac{s-k}{2}] + 2}}\|b\|_{H^{[\frac{s-k}{2}] + 2}}\|\nabla\cdot u\|_{H^{s-k}}
  +\|\p_1\rho\|_{H^{[\frac{s-k}{2}] + 2}} \|\nabla\cdot u\|_{H^{s-k}}\\ & + \|\p_1\rho\|_{H^{s-k}}\|\nabla\cdot u\|_{H^{[\frac{s-k}{2}] + 2}}+\|\rho\|_{H^{s-k}}\|\p_1\nabla\cdot u\|_{H^{s-k}}
\\
  & +\|\p_1b\|_{H^{[\frac{s-k}{2}] + 2}}\|b\|_{H^{[\frac{s-k}{2}] + 2}}\|\nabla u\|_{H^{s-k}}+\|\p_1b\|_{H^{s-k}}\|b\|_{H^{[\frac{s-k}{2}] + 2}}\|\nabla u\|_{H^{[\frac{s-k}{2}] + 2}} \\
  &+ \|b\|_{H^{[\frac{s-k}{2}] + 2}}^2 \|\p_1\nabla u\|_{H^{s-k}} \Big)\|\Omega\|_{H^{s-k}}.
\end{split}
\end{equation}
Indeed, for $k=1,3$ there is,

\begin{equation}\label{Q9}
  \begin{split}
    \int_0^t |Q_9| \; d\tau &
     \lesssim  \;   \mathfrak{E}^\frac{1}{2}(t)\big( \mathcal{E}_k^\frac{1}{2}(t) +\mathcal{F}_k^\frac{1}{2}(t) \big) \mathcal{A}_k^\frac{1}{2}(t) + \mathcal{E}_0^\frac{1}{2}(t)\big( \mathcal{E}_2^\frac{1}{2}(t) +\mathcal{F}_2^\frac{1}{2}(t) \big) \mathcal{A}_k^\frac{1}{2}(t)\\
     &+ \mathfrak{E}(t)\big( \mathcal{E}_k^\frac{1}{2}(t) +\mathcal{F}_k^\frac{1}{2}(t) \big) \mathcal{A}_k^\frac{1}{2}(t)+\mathcal{E}_0^\frac{1}{2}(t)\mathfrak{E}^\frac{1}{2}(t)\big( \mathcal{E}_2^\frac{1}{2}(t) +\mathcal{F}_2^\frac{1}{2}(t) \big) \mathcal{A}_k^\frac{1}{2}(t)\\
    & +\mathfrak{E}(t)\mathcal{B}_k^\frac{1}{2}(t)\mathcal{A}_k^\frac{1}{2}(t)+\mathfrak{E}^\frac{1}{2}(t)\mathcal{E}_0^\frac{1}{2}(t)\mathcal{B}_2^\frac{1}{2}(t)\mathcal{A}_k^\frac{1}{2}(t)
    +\mathfrak{E}^\frac{1}{2}(t)\mathcal{E}_0^\frac{1}{2}(t)\mathcal{E}_2^\frac{1}{2}(t)\mathcal{A}_k^\frac{1}{2}(t).
  \end{split}
\end{equation}

Finally we focus on  the most complicated term $Q_{10}$. By Lemma \ref{prop1},
\begin{align*}
    |Q_{10}|\lesssim & \; w_k \|\Omega\|_{H^{s-k}} \Big\{ \big(1 + \|b\|_{H^{s-k}}\big)\big(\|b_2\|_{H^{[\frac{s-k+1}{2}]+2}} \|\partial_2 u\|_{H^{s-k+1}}\\
    &+ \|b_2\|_{H^{s-k+1}}\|\partial_2 u\|_{H^{[\frac{s-k+1}{2}] + 2}}  + \|b_1\|_{H^{[\frac{s-k+1}{2}]+2}} \|\partial_1 u\|_{H^{s-k+1}}\\
    & + \|b_1\|_{H^{s-k+1}}\|\partial_1 u\|_{H^{[\frac{s-k+1}{2}] + 2}}\big)\\
    & + \|\partial_1 b\|_{H^{s-k}} \big(\|b_2\|_{H^{[\frac{s-k}{2}]+2}} \|\partial_2 u\|_{H^{s-k}} + \|b_2\|_{H^{s-k}}\|\partial_2 u\|_{H^{[\frac{s-k}{2}] + 2}} \\
    & + \|b_1\|_{H^{[\frac{s-k}{2}]+2}} \|\partial_1 u\|_{H^{s-k}} + \|b_1\|_{H^{s-k}}\|\partial_1 u\|_{H^{[\frac{s-k}{2}] + 2}}\big)\Big\}.
  \end{align*}
For $k=1$, we can bound $\int_0^t |Q_{10}| \; d\tau$ as follows

\begin{align*}
    & \quad \sup_{0\leq \tau \leq t}  \big(1 + \|b\|_{H^{s-1}}\big)\|b_2\|_{H^{[\frac{s}{2}]+2}} \int_0^t w_1^\frac{1}{2}\|\partial_2 u\|_{H^{s}}w_1^\frac{1}{2} \|\Omega\|_{H^{s-1}} \; d\tau \\
    & +  \sup_{0\leq \tau \leq t} \big(1 + \|b\|_{H^{s-1}}\big) (1+\tau)^{-1}\|b_2\|_{H^{s}}\int_0^t w_3^\frac{1}{2}\|\partial_2 u\|_{H^{[\frac{s}{2}] + 2}}w_1^\frac{1}{2} \|\Omega\|_{H^{s-1}}  \; d\tau \\
    & +\sup_{0\leq \tau \leq t} \big(1 + \|b\|_{H^{s-1}}\big) w_3^\frac{1}{2} \|b_1\|_{H^{[\frac{s}{2}]+2}} \int_0^t (1+\tau)^{-1}\|\partial_1 u\|_{H^{s}} w_1^\frac{1}{2} \|\Omega\|_{H^{s-1}}  \; d\tau \\
    & + \sup_{0\leq \tau \leq t}  \big(1 + \|b\|_{H^{s-1}}\big) w_1^\frac{1}{2} \|b_1\|_{H^{s}} \int_0^t \|\partial_1 u\|_{H^{[\frac{s}{2}] + 2}} w_1^\frac{1}{2} \|\Omega\|_{H^{s-1}} \; d\tau \\
    & + \sup_{0\leq \tau \leq t} (1+\tau)^{-\frac{1}{2}}\|\partial_1 b\|_{H^{s-1}} \|b_2\|_{H^{[\frac{s-1}{2}]+2}}
    \int_0^t w_2^\frac{1}{2}\|\partial_2 u\|_{H^{s-1}}w_1^\frac{1}{2} \|\Omega\|_{H^{s-1}} \; d\tau \\
    & + \sup_{0\leq \tau \leq t} (1+\tau)^{-1}\|\partial_1 b\|_{H^{s-1}}  \|b_2\|_{H^{s-1}}\int_0^t w_3^\frac{1}{2}\|\partial_2 u\|_{H^{[\frac{s-1}{2}] + 2}} w_1^\frac{1}{2} \|\Omega\|_{H^{s-1}} \; d\tau \\
    & + \sup_{0\leq \tau \leq t} (1+\tau)^{-1}\|\partial_1 b\|_{H^{s-1}} w_3^\frac{1}{2}\|b_1\|_{H^{[\frac{s-1}{2}]+2}} \int_0^t \|\partial_1 u\|_{H^{s-1}}  w_1^\frac{1}{2} \|\Omega\|_{H^{s-1}} \; d\tau \\
    & + \sup_{0\leq \tau \leq t} (1+\tau)^{-\frac{1}{2}}\|\partial_1 b\|_{H^{s-1}} w_2^\frac{1}{2} \|b_1\|_{H^{s-1}}
    \int_0^t \|\partial_1 u\|_{H^{[\frac{s-1}{2}] + 2}}w_1^\frac{1}{2} \|\Omega\|_{H^{s-1}} \; d\tau.
\end{align*}
Here, we have used the fact $(1+t)^{-1} \leq w_0^{\frac{1}{2}}$ and $(1+t)^{-\frac{1}{2}} \leq w_0^{\frac{1}{2}}$.
Therefore,

\begin{equation}\label{1Q10}
\begin{split}
  \int_0^t|Q_{10}| &\; d\tau
\lesssim  \; \mathfrak{E}^\frac{1}{2}(t) \mathcal{E}_1^\frac{1}{2}(t)\mathcal{A}_1^\frac{1}{2}(t) + \mathcal{E}_0^\frac{1}{2}(t)\mathcal{E}_3^\frac{1}{2}(t)\mathcal{A}_1^\frac{1}{2}(t)+\mathcal{E}_3^\frac{1}{2}(t)\mathcal{E}_0^\frac{1}{2}(t)\mathcal{A}_1^\frac{1}{2}(t)\\
&+\mathcal{E}_1^\frac{1}{2}(t)\mathfrak{E}^\frac{1}{2}(t) \mathcal{A}_1^\frac{1}{2}(t)+\mathcal{E}_0^\frac{1}{2}(t)\mathfrak{E}^\frac{1}{2}(t) \mathcal{E}_2^\frac{1}{2}(t) \mathcal{A}_1^\frac{1}{2}(t)+\mathcal{E}_0^\frac{1}{2}(t)\mathfrak{E}^\frac{1}{2}(t) \mathcal{E}_3^\frac{1}{2}(t) \mathcal{A}_1^\frac{1}{2}(t)\\
&+\mathcal{E}_0^\frac{1}{2}(t)\mathcal{E}_3^\frac{1}{2}(t) \mathfrak{E}^\frac{1}{2}(t) \mathcal{A}_1^\frac{1}{2}(t)+\mathcal{E}_0^\frac{1}{2}(t)\mathcal{E}_2^\frac{1}{2}(t) \mathfrak{E}^\frac{1}{2}(t) \mathcal{A}_1^\frac{1}{2}(t).
\end{split}
\end{equation}
Similarly, for $k=3$, we can bound $\int_0^t |Q_{10}| \; d\tau$ as follows.

\begin{align*}
    & \quad \sup_{0\leq \tau \leq t}  \big(1 + \|b\|_{H^{s-3}}\big)\|b_2\|_{H^{s-2}} \int_0^t w_3^\frac{1}{2}\|\partial_2 u\|_{H^{s-2}}w_3^\frac{1}{2} \|\Omega\|_{H^{s-3}} \; d\tau \\
    & +\sup_{0\leq \tau \leq t} \big(1 + \|b\|_{H^{s-3}}\big) w_3^\frac{1}{2} \|b_1\|_{H^{s-2}} \int_0^t \|\partial_1 u\|_{H^{s-2}} w_3^\frac{1}{2} \|\Omega\|_{H^{s-3}}  \; d\tau \\
    & + \sup_{0\leq \tau \leq t} \|\partial_1 b\|_{H^{s-3}} \|b_2\|_{H^{s-3}}
    \int_0^t w_3^\frac{1}{2}\|\partial_2 u\|_{H^{s-3}}w_3^\frac{1}{2} \|\Omega\|_{H^{s-3}} \; d\tau \\
    & + \sup_{0\leq \tau \leq t} \|\partial_1 b\|_{H^{s-3}} w_3^\frac{1}{2}\|b_1\|_{H^{s-3}} \int_0^t \|\partial_1 u\|_{H^{s-3}}  w_3^\frac{1}{2} \|\Omega\|_{H^{s-3}} \; d\tau.
\end{align*}
Which yields that,

\begin{equation}\label{3Q10}
\begin{split}
  \int_0^t|Q_{10}| \; d\tau
\lesssim  \; & \mathfrak{E}^\frac{1}{2}(t) \mathcal{E}_3^\frac{1}{2}(t)\mathcal{A}_3^\frac{1}{2}(t) + \mathcal{E}_3^\frac{1}{2}(t)\mathfrak{E}^\frac{1}{2}(t) \mathcal{A}_3^\frac{1}{2}(t)\\
&+\mathfrak{E}(t) \mathcal{E}_3^\frac{1}{2}(t)\mathcal{A}_3^\frac{1}{2}(t)+\mathfrak{E}^\frac{1}{2}(t) \mathcal{E}_3^\frac{1}{2}(t)\mathfrak{E}^\frac{1}{2}(t)\mathcal{A}_3^\frac{1}{2}(t).
\end{split}
\end{equation}

Integrating \eqref{Q} in time and incorporating the bounds in \eqref{Q1}, \eqref{Q23456}, \eqref{Q7}, \eqref{Q8}, \eqref{Q9}, \eqref{1Q10} and \eqref{3Q10}, applying Young's inequality, we obtain the desired inequality in Proposition \ref{F}.
\end{proof}

\vskip .3in
\section{Decay estimate of $u_2$}
\label{u2}

This section proves the {\it a priori} bound for the energy functional $\mathfrak{U}(t)$.
As indicated by its definition,  $\mathfrak{U}(t)$ measures the regularity and decay properties
of the second velocity component $u_2$. The precise upper bound is stated in the following
proposition.
\begin{lemma}\label{U}
	Let $\mathcal{E}_{total}(t)$ be the total energy defined in \eqref{total}. Then, for any $t>0$, 	
\begin{equation}\nonumber
\mathfrak{U}(t) \lesssim  \mathcal{E}_0(t)+\mathcal{P}_1(t)+\mathcal{B}_1(t) + \mathcal{P}_3(t)+\mathcal{B}_3(t).
\end{equation}
\end{lemma}

\begin{proof}
We recall the equation of $u_2$ in \eqref{mhd1},
\begin{equation}\nonumber
    \partial_t u_2 - \Delta u_2 + \frac{1}{\rho+1} \partial_2 P = - u \cdot \nabla u_2 - \frac{\rho}{\rho + 1} \Delta u_2
    + \frac{1}{\rho + 1} b \cdot \nabla b_2 - \frac{1}{2(\rho + 1)} \partial_2 |b|^2.
\end{equation}
Taking the inner product in $\dot H^{s-2}$ yields
\begin{equation}\label{W}
  \begin{split}
    \frac{1}{2} \frac{d}{dt} w_2\|u_2\|_{\dot H^{s-2}}^2 + w_2\|\nabla u_2\|_{\dot H^{s-2}}^2 = \sum_{i = 1}^3 W_i,
  \end{split}
\end{equation}
where,
\begin{equation}\nonumber
  \begin{split}
    W_1 &= \frac{1}{2}\partial_t w_2 \|u_2\|_{\dot H^{s-2}}^2, \\
    W_2 &= -w_2 \int \p^{s-2}( u\cdot \nabla u_2)\, \p^{s-2} u_2\,dx- w_2 \int \p^{s-2}\Big(\frac{\rho}{\rho + 1} \Delta u_2\Big)\, \p^{s-2} u_2\,dx, \\
    W_3 &= w_2 \int  \p^{s-2}\Big(\frac{1}{\rho + 1} \big(-\partial_2 P + b\cdot \nabla b_2 - \frac{1}{2}\partial_2 |b|^2\big)\Big)\, \p^{s-2} u_2\,dx.
  \end{split}
\end{equation}
To bound $W_1$, we use the fact $w_1 = w_0^\frac{1}{2} w_2^\frac{1}{2}$ to obtain
\begin{equation}\label{W1}
  \begin{split}
    \int_0^t |W_1|\; d\tau = \; \frac{2-\sigma}{2} \int_0^t w_0^\frac{1}{2} \|u_2\|_{\dot H^{s-2}} w_2^\frac{1}{2} \|u_2\|_{\dot H^{s-2}} \; d\tau
    \lesssim  \; \mathcal{E}_0^\frac{1}{2}(t) \mathfrak{U}^\frac{1}{2}(t).
  \end{split}
\end{equation}
It is clear that $W_2$ is bounded by
\begin{equation}\label{W2}
  \begin{split}
    \int_0^t |W_2| \; d\tau
    \lesssim & \sup_{0\leq \tau \leq t} \|u\|_{H^{s-2}} \int_0^t w_2 \|\nabla u_2\|_{\dot H^{s-2}}^2 \; d\tau\\
    & +
    \sup_{0\leq \tau \leq t} \|\rho\|_{H^{s-3}} \int_0^t w_2 \|\Delta u_2\|_{H^{s-3}} \|u_2\|_{\dot H^{s-1}}\; d\tau \\
    \lesssim & \; \mathfrak{E}^\frac{1}{2}(t) \mathfrak{U}(t).
  \end{split}
\end{equation}
By integration by parts,

\begin{equation}\nonumber
  \begin{split}
     |W_3|
    \lesssim  \; w_2 \|\rho\|_{H^{s-3}}\big(\|\partial_2 \rho\|_{H^{s-3}} + \|\nabla b\|_{H^{s-3}} \| b\|_{H^{s-3}}    \big) \|u_2\|_{\dot H^{s-1}}.
  \end{split}
\end{equation}
Therefore, notice the assumption \ref{assum} we shall derive

\begin{equation}\label{W3}
  \begin{split}
    \int_0^t |W_3| \; d\tau \lesssim \; & \mathfrak{E}^\frac{1}{2}(t)\mathcal{P}_2^\frac{1}{2}(t)\mathfrak{U}^\frac{1}{2}(t)+\mathfrak{E}(t)\mathcal{B}_2^\frac{1}{2}(t)\mathfrak{U}^\frac{1}{2}(t) \\
    \lesssim \; & \big( \mathcal{P}_1^\frac{1}{2}(t)+\mathcal{B}_1^\frac{1}{2}(t)+\mathcal{P}_3^\frac{1}{2}(t)+\mathcal{B}_3^\frac{1}{2}(t)\big)\mathfrak{U}^\frac{1}{2}(t).
  \end{split}
\end{equation}
Integrating \eqref{W} in time, applying Young's inequality, and using \eqref{W1}, \eqref{W2} and \eqref{W3}, we then obtain the desired inequality stated.
\end{proof}

\vskip .3in
\section{Proof of Theorem \ref{thm1}}
\label{pp}

This section makes use of the estimates in the propositions of the previous sections
to finish the proof of Theorem \ref{thm1}.

\begin{proof}[Proof of Theorem \ref{thm1}] The local-in-time well-posedness of \eqref{mhd1} in the Sobolev setting $H^s(\mathbb T^2)$ with sufficiently large $s$ can be shown following standard approaches (see, e.g., \cite{MaBe}). It remains to establish the global bound of $(\rho, u, b)$ in $H^s$. We employ the bootstrapping argument (see,  e.g.,\cite{Tao}). The first step here is
to combine the inequalities obtained in the propositions for $\mathcal{E}_0(t)$, $\mathfrak{E}(t)$, $\mathcal{E}_k$, $\mathcal{P}_k$, $\mathcal{B}_k$, $\tilde{\mathcal{A}}(t)$, $\mathcal{A}_k(t)$ and $\mathfrak{U}(t)$ to show that, for some positive constant $C^*$,
\begin{equation}\label{ee}
	\mathcal{E}_{total}(t) \leq C^* \mathcal{E}_{total}(0)+ C^* \mathcal{E}^{\frac{3}{2}}_{total}(t),
\end{equation}
where $\mathcal{E}_{total}(t)$ is as defined in (\ref{total}), or
\begin{equation}\nonumber
\begin{split}
  \mathcal{E}_{total}(t)\triangleq & \sum_{i=0,1,3} \mathcal{E}_i(t)+\sum_{i=1,3} \mathcal{P}_i(t)+\sum_{i=1,3} \mathcal{B}_i(t)+\sum_{i=1,3} \mathcal{F}_i(t)+\sum_{i=1,3} \mathcal{A}_i(t) \\
 & \qquad  +  \tilde{\mathcal{A}}(t)  +  \mathfrak{U}(t)+\mathfrak{E}(t).
\end{split}
\end{equation}
According to Propositions \ref{lemmaBasic}, \ref{strongDis},  \ref{strongDis2}, \ref{strongDis3},
\ref{A}, \ref{F} and \ref{U},
\begin{align}
 &\mathcal{E}_0(t)  \lesssim  \; \mathcal{E}_0(0) +  \mathcal{E}_{total}^{\frac{3}{2}}(t), \label{b1}\\
 &\mathfrak{E}(t)  \lesssim \; \mathcal{E}_0(0) +  \mathcal{E}_{total}^{\frac{3}{2}}(t),
 \label{b2}\\
& \mathcal{E}_1(t) \lesssim \; \mathcal{E}_1(0)+\mathcal{E}_0^{\frac{1}{2}} (t) \big( \mathcal{E}_1^{\frac{1}{2}}(t) + \mathcal{P}_1^{\frac{1}{2}}(t) +\mathcal{B}_1^{\frac{1}{2}}(t)  \big)+ \mathcal{E}^{\frac{3}{2}}_{total}(t), \label{b3}\\
 &\mathcal{E}_3(t) \lesssim \;  \mathcal{E}_3(0)+\big( \mathcal{E}_1^{\frac{1}{2}}(t) + \mathcal{P}_1^{\frac{1}{2}} (t)+\mathcal{B}_1^{\frac{1}{2}} (t) \big)\big( \mathcal{E}_3^{\frac{1}{2}}(t) + \mathcal{P}_3^{\frac{1}{2}}(t) +\mathcal{B}_3^{\frac{1}{2}} (t) \big) + \mathcal{E}^{\frac{3}{2}}_{total}(t),\label{b4}\\
 &\mathcal{P}_k(t) \lesssim \;\mathcal{E}_k(0)+ \mathcal{E}_k(t)+\mathcal{E}^{\frac{3}{2}}_{total}(t),\qquad k=1, 3,\label{b5}\\
&\mathcal{B}_k(t) \lesssim \;\mathcal{E}_k(0)+ \mathcal{E}_k(t)+\mathcal{P}_k(t)+\mathcal{E}^{\frac{3}{2}}_{total}(t),\qquad k=1, 3,\label{b6}\\
&\tilde{\mathcal{A}}(t)  \lesssim \; \mathcal{E}_0 (t) +  \mathcal{E}^{\frac{3}{2}}_{total}(t),
\label{b7}\\
&\mathcal{A}_k(t) \lesssim  \;  \mathcal{E}_k (t) +\mathcal{F}_k(t) +\mathcal{E}^{\frac{3}{2}}_{total}(t),  \qquad  k=1,3,\label{b8}\\
&\mathcal{F}_k(t) \lesssim \;   \mathcal{E}_0(t)+ \tilde{\mathcal{A}}(t)  +\mathcal{E}_k^\frac{1}{2} (t) \mathcal{A}_k^\frac{1}{2}(t) + \mathcal{E}_{total}^{\frac{3}{2}}(t) , \qquad k=1,3,\label{b9}\\
&\mathfrak{U}(t) \lesssim  \mathcal{E}_0(t)+\mathcal{P}_1(t)+\mathcal{B}_1(t) + \mathcal{P}_3(t)+\mathcal{B}_3(t). \label{b10}
\end{align}
(\ref{b1}) and (\ref{b2}) imply that $\mathcal{E}_0(t)$ and $\mathfrak{E}(t)$ satisfies (\ref{ee}).
(\ref{b7}) then assesses that $\tilde{\mathcal{A}}(t)$ satisfies (\ref{ee}). Inserting (\ref{b1}), and (\ref{b5}) and (\ref{b6}) with $k=1$ in (\ref{b3}) yields

\begin{align*}
\mathcal{E}_1(t) \lesssim&  \;\mathcal{E}_1(0) + \big(\mathcal{E}_0^{\frac12}(0) +  \mathcal{E}_{total}^{\frac{3}{4}}(t)\big)
\big(\mathcal{E}_0^{\frac12}(0) + \mathcal{E}_1^{\frac12}(t) +  \mathcal{E}_{total}^{\frac{3}{4}}(t)\big)  + \mathcal{E}^{\frac{3}{2}}_{total}(t)\\
\le& \; \frac12 \mathcal{E}_1(t) + C\, \mathcal{E}_{total}(0)  + C\, \mathcal{E}^{\frac{3}{2}}_{total}(t).
\end{align*}
That is,
\beq\label{b11}
\mathcal{E}_1(t) \le C\, \mathcal{E}_{total}(0) + C\, \mathcal{E}^{\frac{3}{2}}_{total}(t).
\eeq
Inserting (\ref{b11}) in (\ref{b5}) and (\ref{b6}) with $k=1$ leads to
\begin{align}
	&\mathcal{P}_1(t) \le C\,  \mathcal{E}_{total}(0)+ C\, \mathcal{E}^{\frac{3}{2}}_{total}(t), \label{b12} \\
	&\mathcal{B}_1(t) \le C\,  \mathcal{E}_{total}(0) + C\, \mathcal{E}^{\frac{3}{2}}_{total}(t). \label{b13}
\end{align}
Inserting (\ref{b11}), (\ref{b12}) and (\ref{b13}) as well as (\ref{b5}) and (\ref{b6}) with $k=3$
in (\ref{b4}), we find that
$$
\mathcal{E}_3(t) \le C\,  \mathcal{E}_{total}(0) + C\, \mathcal{E}^{\frac{3}{2}}_{total}(t),
$$
which, in turn, yields
\begin{align*}
	&\mathcal{P}_3(t) \le C\,  \mathcal{E}_{total}(0) + C\, \mathcal{E}^{\frac{3}{2}}_{total}(t),  \\
	&\mathcal{B}_3(t) \le C\,  \mathcal{E}_{total}(0) + C\, \mathcal{E}^{\frac{3}{2}}_{total}(t).
\end{align*}
As a consequence, (\ref{b10}) implies $\mathfrak{U}(t)$ satisfies (\ref{ee}).
Inserting (\ref{b8}) in (\ref{b9}) then shows that
$$
\mathcal{F}_k(t) \le C\, \mathcal{E}_{total}(0) + C\, \mathcal{E}^{\frac{3}{2}}_{total}(t),
$$
which subsequently implies
$$
\mathcal{A}_k(t)  \le C\, \mathcal{E}_{total}(0) + C\, \mathcal{E}^{\frac{3}{2}}_{total}(t).
$$
We have thus verifies (\ref{ee}).

\vskip .1in
We apply the bootstrapping argument to (\ref{ee}). Assume the initial data $(\rho_0, u_0, b_0)$ in $H^s(\mathbb{T}^2)$ is sufficiently small,
$$
\|(\rho_0, u_0, b_0)\|_{H^s} \le \epsilon
$$
such that
$$
 C^* \mathcal{E}_{total}(0) \le \frac{1}{16 (C^*)^2}.
$$
If we make the ansatz that
\beq\label{an}
\mathcal{E}_{total}(t) \le \frac1{4 (C^*)^2},
\eeq
then (\ref{ee}) implies
$$
\mathcal{E}_{total}(t) \le \frac{1}{16 (C^*)^2} + \frac12 \mathcal{E}_{total}(t)
$$
or
\beq\label{cc}
\mathcal{E}_{total}(t) \le \frac1{8 (C^*)^2}.
\eeq
That is, (\ref{cc}) provides a sharper upper bound than the ansatz (\ref{an}) does. The bootstrapping
argument then concludes that (\ref{cc}) holds for any $t>0$. The desired global bound of
$(\rho, u, b)$ in $H^s(\mathbb{T}^2)$ then follows. This finishes the proof of Theorem \ref{thm1}.
Notice that the constant $C^*$ above can be chosen large enough so that $\mathcal{E}_{total}(t) < \frac{1}{2}$, it then shows the rationality of assumption \ref{assum}.
\end{proof}

\vskip .2in
\section*{Acknowledgement}
Wu was partially supported by the National Science Foundation of the
United States under grant DMS 2104682 and the AT\&T Foundation at Oklahoma State
University. Zhu was partially supported by Natural Science Foundation of Shanghai (22ZR1418700).

\end{document}